\documentclass[a4paper,12pt,reqno,oneside]{amsart} 

\usepackage{setspace} 
\usepackage{typearea} 
\usepackage{amscd,amssymb,verbatim} 
\usepackage
{graphicx} \setkeys{Gin}{width=\linewidth}
\usepackage{enumitem}
\usepackage{bbold}

\usepackage{xr-hyper}
\usepackage{cite} 
\usepackage[colorlinks,backref,final,bookmarksnumbered,bookmarks]{hyperref}
\usepackage[backrefs]{amsrefs}

\usepackage{ifthen}
\newcommand{\arxiv}[2][]{\ifthenelse{\equal{#1}{}}
{\href{http://arxiv.org/abs/#2}{\tt arXiv:#2}}
{\href{http://arxiv.org/abs/math/#2}{\tt arXiv:math.#1/#2}}}

\usepackage[T1,T2A]{fontenc}
\usepackage[utf8]{inputenc}

\makeatletter
\renewcommand\subsection{\@startsection
{subsection}{2}{0cm} 
{-\baselineskip}     
{0.5\baselineskip}   
{\sffamily}} 
\makeatother

\theoremstyle{plain}
\newtheorem{theorem}{Theorem}[section]

\newtheorem{corollary}[theorem]{Corollary}
\newtheorem{proposition}[theorem]{Proposition}

\theoremstyle{definition}
\newtheorem{example}[theorem]{Example}

\newtheoremstyle{remark}
{}{}{}{}{\itshape}{}{ }{\thmname{#1}\thmnumber{ \itshape #2.}}
\theoremstyle{remark}
\newtheorem{remark}[theorem]{Remark}

\newtheoremstyle{concise}
{}{}{}{}{\bfseries}{}{ }{\thmnumber{#2.}\thmnote{ #3.}}
\theoremstyle{concise}

\DeclareFontEncoding{LS1}{}{}
\DeclareFontSubstitution{LS1}{stix}{m}{n}
\DeclareFontEncoding{LS2}{}{}
\DeclareFontSubstitution{LS2}{stix}{m}{n}

\DeclareMathVersion{language}
\DeclareMathVersion{metalanguage}

\def\formulas{\mathversion{language}}

\def\metameta{\mathversion{normal}}

\DeclareMathAlphabet{\mf}{OML}{zplm}{m}{it} 
\DeclareMathAlphabet{\ts}{OT1}{cmss}{m}{sl} 

\DeclareMathAlphabet{\fm}{U}{eur}{m}{n} 
\DeclareMathAlphabet{\tr}{OT1}{cmss}{m}{n} 
\DeclareMathAlphabet{\mm}{OML}{cmm}{m}{it} 

\DeclareMathAlphabet{\script}{LS1}{stixscr}{m}{n}
\DeclareMathAlphabet{\mathbbb}{U}{bbold}{m}{n}

\SetSymbolFont{letters}{language}{U}{eur}{m}{n}
\SetSymbolFont{letters}{metalanguage}{OML}{zplm}{m}{it}

\newcommand{\newsym}[5]{\fontfamily{#2}\fontencoding{#1}\fontseries{#3}\fontshape{#4}\selectfont\char#5}
\makeatletter
\newcommand{\newmathsymbol}[6]{#1{\@Pimathsymbol{#2}{#3}{#4}{#5}{#6}}}
\def\@Pimathsymbol#1#2#3#4#5{\mathchoice
  {\@Pim@thsymbol{#1}{#2}{#3}{#4}{#5}\tf@size}
  {\@Pim@thsymbol{#1}{#2}{#3}{#4}{#5}\tf@size}
  {\@Pim@thsymbol{#1}{#2}{#3}{#4}{#5}\sf@size}
  {\@Pim@thsymbol{#1}{#2}{#3}{#4}{#5}\ssf@size}}
\def\@Pim@thsymbol#1#2#3#4#5#6{\mbox{\fontsize{#6}{#6}\newsym{#1}{#2}{#3}{#4}{#5}}}
\makeatother

\def\too{\newmathsymbol{\mathrel}{LS1}{stixsf}{m}{n}{"99}}
\newcommand{\Dfrac}[2]{\genfrac{}{}{1.2pt}0{#1}{#2}}

\def\lgroup{\newmathsymbol{\mathopen}{LS2}{stixex}{m}{n}{"DC}}
\def\rgroup{\newmathsymbol{\mathclose}{LS2}{stixex}{m}{n}{"DD}}
\newcommand{\mq}[1]{\lgroup #1\rgroup\,}

\def\mc#1{\mq{\oldstylenums{#1}}}

\def\prin{\boldsymbol\cdot\hskip1.5pt}

\def\cltop{\top}
\def\clbot{\bot}
\def\ab{{\mathchoice
  {\mbox{\larger[1]$\times$}}
  {\mbox{\larger[1]$\times$}}
  {\mbox{\larger[-2]$\times$}}
  {\mbox{\larger[-4]$\times$}}
}}
\def\triv{\checkmark}
\def\from{\leftarrow}
\def\To{\ \longrightarrow\ }
\def\tofrom{\leftrightarrow}
\def\Tofrom{\ \longleftrightarrow\ }
\def\imp{\Rightarrow}
\def\Imp{\ \Longrightarrow\ }

\def\iff{\Leftrightarrow}
\def\Iff{\ \Longleftrightarrow\ }

\def\turnstile{\vdash}

\DeclareMathSymbol{\impord}{\mathord}{symbols}{41}
\DeclareMathSymbol{\mand}{\mathbin}{operators}{`\&}

\DeclareMathSymbol{\oc}{\mathord}{operators}{`!}
\DeclareMathSymbol{\wn}{\mathord}{operators}{`?}

\DeclareMathSymbol{\alpha}     {\mathalpha}{letters}{"0B}
\DeclareMathSymbol{\beta}      {\mathalpha}{letters}{"0C}
\DeclareMathSymbol{\gamma}     {\mathalpha}{letters}{"0D}
\DeclareMathSymbol{\delta}     {\mathalpha}{letters}{"0E}
\DeclareMathSymbol{\epsilon}   {\mathalpha}{letters}{"0F}
\DeclareMathSymbol{\zeta}      {\mathalpha}{letters}{"10}
\DeclareMathSymbol{\eta}       {\mathalpha}{letters}{"11}
\DeclareMathSymbol{\theta}     {\mathalpha}{letters}{"12}
\DeclareMathSymbol{\iota}      {\mathalpha}{letters}{"13}
\DeclareMathSymbol{\kappa}     {\mathalpha}{letters}{"14}
\DeclareMathSymbol{\lambda}    {\mathalpha}{letters}{"15}
\DeclareMathSymbol{\mu}        {\mathalpha}{letters}{"16}
\DeclareMathSymbol{\nu}        {\mathalpha}{letters}{"17}
\DeclareMathSymbol{\xi}        {\mathalpha}{letters}{"18}
\DeclareMathSymbol{\pi}        {\mathalpha}{letters}{"19}
\DeclareMathSymbol{\rho}       {\mathalpha}{letters}{"1A}
\DeclareMathSymbol{\sigma}     {\mathalpha}{letters}{"1B}
\DeclareMathSymbol{\tau}       {\mathalpha}{letters}{"1C}
\DeclareMathSymbol{\upsilon}   {\mathalpha}{letters}{"1D}
\DeclareMathSymbol{\phi}       {\mathalpha}{letters}{"1E}
\DeclareMathSymbol{\chi}       {\mathalpha}{letters}{"1F}
\DeclareMathSymbol{\psi}       {\mathalpha}{letters}{"20}
\DeclareMathSymbol{\omega}     {\mathalpha}{letters}{"21}
\DeclareMathSymbol{\varepsilon}{\mathalpha}{letters}{"22}
\DeclareMathSymbol{\vartheta}  {\mathalpha}{letters}{"23}
\DeclareMathSymbol{\varpi}     {\mathalpha}{letters}{"24}
\DeclareMathSymbol{\varrho}    {\mathalpha}{letters}{"25}
\DeclareMathSymbol{\varsigma}  {\mathalpha}{letters}{"26}
\DeclareMathSymbol{\varphi}    {\mathalpha}{letters}{"27}

\DeclareMathSymbol{\bbomega}{\mathalpha}{letters}{"7F}

\def\Ds{\mathcal{D}}
\def\Dec{\mathfrak{D}}
\def\F{\mathcal{F}}

\def\G{\mathcal{G}}
\def\H{\mathcal{H}}

\def\L{\script{L}}

\def\N{\mathbbb{N}}

\def\Q{\mathbbb{Q}}

\def\R{\mathbbb{R}}
\def\Rt{\mathbbb{\omega}}

\def\Stab{\mathfrak{S}}

\def\0{\mathbbb{0}}
\def\1{\mathbbb{1}}
\def\m{\mathbbb{\mu}}
\def\q{\mathbbb{q}}

\def\var{\mathbbb{x}}
\def\pr{\mathbbb{p}}

\def\x{\times}

\def\phi{\varphi}
\def\emptyset{\varnothing}

\renewcommand{\:}{\colon}

\DeclareMathOperator{\Int}{Int}
\DeclareMathOperator{\Cl}{Cl}

\DeclareMathOperator{\Hom}{Hom}

\DeclareMathOperator{\Plus}{+}
\DeclareMathOperator{\Neg}{\sim}

\DeclareMathOperator\suchthat{:}

\def\octo{\oc_{_\to}} \def\wnto{\wn_{_\to}}

\begin{document}

\title[A Galois connection between intuitionistic \& classical logics. I: Syntax]
{A Galois connection between intuitionistic and classical logics. I: Syntax}
\addtocounter{footnote}{1}
\footnotetext{{\it Galois connection} is a standard notion of order theory, whose eponymous example
is the correspondence between the poset of fixed fields and the poset of subgroups in Galois theory.
It can be defined as a pair of adjoint functors between two posets, regarded as categories.
See \cite{HH}*{pp.\ 166--167} for a concise introduction to Galois connections,
and \cite{EKMS} for further details.}
\author{Sergey A. Melikhov}
\address{Steklov Mathematical Institute of the Russian Academy of Sciences,
ul.\ Gubkina 8, Moscow, 119991 Russia}
\email{melikhov@mi-ras.ru}
\thanks{Supported by The Fund for Math and
Russian Foundation for Basic Research Grant No.\ 15-01-06302}

\begin{abstract}
In a 1985 commentary to his collected works \cite{Kol3},
Kolmogorov remarked that his 1932 paper \cite{Kol}
``was written in hope that with time, the logic of solution of problems
[i.e., intuitionistic logic] will become a permanent part of a [standard]
course of logic.
A unified logical apparatus was intended to be created, which would deal
with objects of two types --- propositions and problems.''
We construct such a formal system, as well as its predicate version QHC, 
which is a conservative extension of both the intuitionistic predicate calculus QH 
and the classical predicate calculus QC.

The only new connectives $\wn$ and $\oc$ of QHC induce a Galois connection
between the Lindenbaum posets (i.e.\ the underlying posets of the Lindenbaum algebras)
of QH and QC.
Kolmogorov's double negation translation of propositions into problems
extends to a retraction of QHC onto QH; whereas G\"odel's provability translation of problems
into modal propositions
extends to a retraction of QHC onto its QC+($\wn\oc$) fragment, identified with the modal logic QS4.
The QH+($\oc\wn$) fragment is an intuitionistic modal logic, whose modality $\oc\wn$ is a
``strict lax modality'' in the sense of Aczel --- and thus resembles the squash/bracket operation in
intuitionistic type theories.

The axioms of QHC attempt to give a fuller formalization (with respect to the axioms of intuitionistic
logic) to the two best known contentual interpretations of intiuitionistic logic: Kolmogorov's problem
interpretation (incorporating standard refinements by Heyting and Kreisel) and the proof interpretation
by Orlov and Heyting (as clarified by G\"odel).
While these two interpretations are often conflated, from the viewpoint of the axioms of QHC neither
of them reduces to the other one, although they do overlap.
\end{abstract}

\maketitle

\section{Introduction}\label{intro}

\subsection{Problems versus propositions}\label{intro1}
The present series of papers (the sequels being \cite{M2} and \cite{M3}) belongs firmly to the field of Logic,
but is motivated primarily by considerations of mathematical practice rather than any internal developments
in the field of Logic.
Therefore it is addressed not only to logicians, but to other mathematicians as well.
The reader who is not familiar with any of the terms used can consult the treatise \cite{M0} as need arises;
one of its main goals is precisely to make the present series accessible to a general mathematical audience.

This paper introduces a logical apparatus that enables one to study in a formal setting basic
interdependencies between what can be called (cf.\ \S\ref{philosophy}) two modes of
knowledge: knowledge-that (or knowledge of truths) and knowledge-how (or knowledge of methods).
In mathematical practice, these have been traditionally represented by {\it propositions} (i.e., assertions,
such as theorems and conjectures) and {\it problems} (such as geometric construction problems and
initial value problems).
The English word ``problem'' is, in fact, somewhat imprecise; we will use it in the narrow sense of
a request (or desire) to find a construction meeting specified criteria on output and permitted means
(as in ``chess problem'').
This meaning is less ambiguously captured by the German {\it Aufgabe} (as opposed to the German
{\it Problem}) and the Russian {\it задача} (as opposed to {\it проблема}).
The closest English word is {\it task} (other words with related meanings include {\it assignment, exercise,
challenge, aim, mission}), but as it is not normally used in mathematical contexts, we prefer to speak
of {\it problems}.

To appreciate the difference between problems and propositions, let us note firstly that the problem
requesting to find a proof of a proposition $P$ is closely related to both (i) the proposition
asserting that $P$ is true; and (ii) the proposition asserting that $P$ is provable.
These are not the same, of course, whenever ``proofs'' are taken to be in some formal theory $T$ and
``truth'' is taken according to some two-valued model $M$ of $T$, with respect to which $T$ is not
complete.%
\footnote{For instance, $T$ could consist of the axioms of planar geometry except for the axiom of
parallel lines, and $M$ could be the Euclidean planar geometry.
One could object that it would be fair to compare truth according to Euclidean geometry with proofs
in its complete theory; but then $T$ could be Peano Arithmetic or ZFC, which by G\"odel's theorem are
not complete with respect to any models. (See \cite{M0}*{\S\ref{int:conjectures}} for a more detailed
discussion.) }
There are other reasons why ``true'' should not be equated with ``provable''; for instance, they differ
also in the modal logic S4, where it is a simple consequence of the axioms that consistency is provable.%
\footnote{With respect to the internal notion of provability. As shown by Art\"emov \cite{Ar1}, the latter can
be modelled by the existence of proofs in Peano Arithmetic, where ``proofs'' have the usual meaning of
formal proofs (except that Art\"emov needs one ``proof'' to be able to prove several formulas),
but ``existence'' is understood in an explicit sense, not expressible internally in Peano Arithmetic.}
Conversely, the proposition asserting that two groups $G$ and $H$ are isomorphic is closely related to both
(i) the problem requesting to prove that $G$ and $H$ are isomorphic; and (ii) the problem requesting to
construct an isomorphism between $G$ and $H$.
These are generally not the same because one proof that an isomorphism exists might represent
several distinct isomorphisms or no specific isomorphism.

The logical distinction between problems and theorems (as they appear, in particular, in Euclid's
{\it Elements}) has been articulated at length by a number of ancient Greek geometers in response to
others who disputed it.
A detailed review of what the ancients had to say on this matter is included in the third part of this paper
\cite{M3}.
In modern times, the distinction was emphasized by Kolmogorov \cite{Kol}:

\smallskip
\begin{center}
\parbox{14.7cm}{\small
``On a par with theoretical logic, which systematizes schemes of
proofs of theoretical truths, one can systematize schemes of
solutions of problems --- for example, of geometric construction
problems.
For instance, similarly to the principle of syllogism we have
the following principle here:
{\it If we can reduce solving $b$ to solving $a$, and
solving $c$ to solving $b$, then we can also reduce solving $c$ to
solving $a$.}

Upon introducing appropriate notation, one can specify the rules of
a formal calculus that yield a symbolic construction of a system of
such problem solving schemes.
Thus, in addition to theoretical logic, a certain new
{\it calculus of problems} arises.
In this setting there is no need for any special, e.g.\ intuitionistic,
epistemic presuppositions.

The following striking fact holds: {\it The calculus of problems
coincides in form with Brouwer's intuitionistic logic,
as recently formalized by Mr.\ Heyting}.

In the second section we undertake a critical analysis of intuitionistic logic, accepting
general intuitionistic presuppositions;
and observe that intuitionistic logic should be replaced with the calculus of problems,
since its objects are in reality not theoretical propositions but
rather problems.''
}
\end{center}
\medskip

A key difference between problems and propositions is that the notion
of truth for propositions has no direct analogue for problems, so that
problems cannot be asserted.
For instance, let $\Gamma$ be the problem {\it Divide any given angle
into three equal parts with compass and (unmarked) ruler}.
Then $\Gamma\lor\neg\Gamma$ reads, {\it Divide
any given angle into three equal parts with compass and ruler or prove
that it is impossible to do so} (cf.\ \cite{M0}*{\S\ref{int:about-bhk}} and
\ref{insolubility} below).
This is not a trivial problem; indeed, its solution took a couple of
millennia.
Even now that a solution is well-known, the problem still makes perfect
sense: the law of excluded middle would not help a student to solve this
problem on an exam (in Galois theory).
By citing the law of excluded middle she could solve another problem:
{\it Prove that either $\Gamma$ has a solution or $\Gamma$ has no solutions;} in symbols,
$$\oc(\wn\Gamma\lor\neg\wn\Gamma),$$
where $\wn\Gamma$ denotes the proposition {\it There exists a solution of the problem
$\Gamma$}, and $\oc P$ denotes the problem {\it Prove the proposition $P$}.%
\footnote{Let us explain the notation.
A proposition comes with a question whether it is true or false;
whereas a problem comes with an urge to solve it.
Thus $\wn$ can serve as a concise typing symbol for propositions, and
$\oc$ for problems.
By placing a typing symbol in front of a sentence we indicate its
conversion into the corresponding type.}
In fact, this problem is strictly easier than $$\oc\wn\Gamma\lor\oc\neg\wn\Gamma$$
(in words, {\it Prove or disprove that $\Gamma$ has a solution}), which requires a justified explicit choice.
But the latter problem, which can be written equivalently as $\oc\wn\Gamma\lor\neg\Gamma$,
is still strictly easier than the original problem,
$\Gamma\lor\neg\Gamma$, for it is generally easier to prove that some problem has a solution
than to actually solve it.

\subsection{A joint logic}
The present paper is devoted to the study of the logical operators $\wn$
and $\oc$ in a formal setting.
Like in the previous example, $\oc$ is meant to refer to {\it non-constructive} proofs, whereas
$\wn$ is understood to signify {\it explicit} existence.
We extract axioms and rules governing the use of $\wn$ and $\oc$ essentially from
two sources:
\begin{itemize}
\item The problem interpretation of intuitionistic logic.
This is essentially Kolmogorov's 1932 explanation of the intuitionistic connectives \cite{Kol}, which had
some parallels with the independent writings of Heyting (1931), and was slightly refined by Heyting (1934).
A disguised form of this explanation, often incorporating a further refinement by Kreisel, has come to be
known as the BHK interpretation of intuitionistic logic (see \cite{M0}*{\S\ref{int:intro}} for a detailed
review and discussion).
We include Kreisel's addendum in the following form, also found in the ancient commentary by Proclus
on Euclid's Elements (see \cite{M3}): A solution of a problem must include not only a construction,
but also the verification, i.e.\ a proof that the construction meets the requirements specified in
the problem (see \cite{M0}*{\S\ref{int:about-bhk}} for a discussion of this principle).
\item The proof interpretation of intuitionistic logic.
This is essentially the meaning explanation of intuitionistic logic given independently by Orlov (1928)
and Heyting (1930, 31) (see a detailed review in \S\ref{letters1}), which was partially formalized in
G\"odel's 1933 translation of problems into modal propositions (see \cite{M0}*{\S\ref{int:provability}}),
and further clarified by G\"odel's proof-relevant analogue of S4 (see \S\ref{Goedel}).%
\footnote{The fact that the Orlov--Heyting--G\"odel proof interpretation is substantially different from
the Kolmogorov--Heyting--Kreisel problem interpretation seems to have been never properly recognized,
except that G\"odel's paper formalizing the Orlov--Heyting interpretation begins with a reference
to Kolmogorov's ``somewhat different interpretation ... [given] without, to be sure,
specifying a precise formalism'' \cite{Goe2} (see also \cite{Tr99}*{p.\ 235}).
A certain precise formalism attempting to capture Kolmogorov's interpretation alone is specified in
\cite{M0}*{\S\ref{int:weak BHK}}.}
\end{itemize}

It then comes as a little surprise that the resulting axioms and rules harbor
a great deal of unintended symmetries, and are also compatible with Kolmogorov's
double negation translation of propositions into problems (reviewed
briefly in \cite{M0}*{\S\ref{int:negneg}}).
(For a different connection between Kolmogorov's and G\"odel's translations see \cite{D86}.)
What is most surprising, however, is that nobody seems to have studied
the operators $\wn$ and $\oc$ before, apart from hints of an abandoned
project aimed at a similar study, found in Kolmogorov's own writings.
In his 1931 letter to Heyting \cite{Kol2}, Kolmogorov wrote:

\smallskip
\begin{center}
\parbox{14.7cm}{\small
Each `proposition'
in your framework belongs, in my view, to one of two sorts:
\begin{itemize}
\item[($\alpha$)] $p$ expresses hope that in prescribed circumstances,
a certain experiment will always produce a specified result.
(For example, that an attempt to represent an even number $n$
as a sum of two primes will succeed upon exhausting all pairs
$(p,q)$, $p<n$, $q<n$.\footnotemark)
Of course, every ``experiment'' must be realizable by a finite
number of deterministic operations.
\item[($\beta$)] $p$ expresses intention to find a certain construction.
\end{itemize}
[...] I prefer to keep the name {\it proposition} (Aussage) only for
propositions of type ($\alpha$) and to call ``propositions'' of
type ($\beta$) simply {\it problems} (Aufgaben). Associated to
a {\it proposition} $p$ are the {\it problems} $\Neg p$ (to derive
contradiction from $p$) and $\Plus p$ (to prove $p$).
}\footnotetext{This $p$ (prime number) is unrelated to the previous $p$ (proposition).}
\end{center}
\smallskip

Apart from this fragment and the quote in the abstract, there are only a few further hints
at how Kolmogorov envisaged the connection between problems and propositions.
Several problems consisting in proving a proposition are also mentioned in Kolmogorov's paper
\cite{Kol}.
There is also a bit more in Kolmogorov's letters to Heyting, which will be thoroughly reviewed
in \S\ref{letters2}.
There we note, in particular, that while Kolmogorov's propositions of type ($\beta$) seem to
stand precisely for the objects of intuitionistic logic, his propositions of type ($\alpha$)
could not be intended to exhaust all objects of classical logic; in fact, it appears
that they can be identified with the ``stable propositions'' of \S\ref{stable and decidable}.
Another apparent divergence between Kolmogorov's remarks and our approach is noted in
\ref{insolubility} and discussed more thoroughly in \cite{M0}*{\S\ref{int:about-bhk}}.

The joint logic of problems and propositions that is constructed in the present paper is
presumably very unnatural in the standard
constructivist paradigm (of Brouwer and Heyting) that views intuitionistic logic as an alternative to
classical logic that criminalizes some of its principles.
We work in the other paradigm (of Kolmogorov), which views intuitionistic logic as
an extension package that upgrades classical logic without removing it.
For us, the main purpose of this upgrade is solution-relevance (=``proof-relevance''), or
``categorification''.
Thus from the viewpoint of the BHK semantics, topological (Tarski) models are in fact models of
a ``squashed'' copy of intuitionistic logic --- whose existence is only revealed with the aid of
the new connectives $\oc$ and $\wn$ (see \S\ref{nabla} below); whereas ``true'' models of
the genuine intuitionistic logic are the (solution-relevant) sheaf-valued models of \cite{M0}
(a special case of ``categorical models'' --- not to be confused with the usual ``sheaf models''
of intuitionistic logic).
Models of the joint logic of problems and propositions will be discussed in \cite{M3}.

\subsection{Double negation translation}

Speaking of ``intuitionistic logic as an extension package that upgrades classical logic
without removing it'', we run into the natural question: ``Wait, but what about the double negation
translation?''
Indeed, there is a version of the double negation translation that redefines classical connectives in
terms of intuitionistic ones and introduces no other modifications to formulas (see
\cite{M0}*{\S\ref{int:negneg}}).
However, this syntactic translation fails to reflect actual mathematical practice.
There are several levels at which this failure occurs:

(i) In the words of Kreisel \cite{Kr73}, ``there is a {\it good} reason why
mathematicians neglect'' the double negation translation, in the form of ``replacing $\exists$ by
$\neg\forall\neg$ and $p\lor q$ by $\neg(\neg p\lor\neg q)$'', ``namely, this: {\it For the sense
in which mathematicians actually understand the propositions of mathematical practice, ...
the difference between $\exists$ and $\lor$ on the one hand and their translations on the other
... is not significant}''.
``Put differently, they do not understand the intuitionistic meaning of $\forall$ and $\neg$ which
makes the [double negation] translation significant.''

This is not merely a matter of mathematicians' conventions, psychology or ignorance.
For mathematicians to be serious about the intuitionistic meaning of propositions, in the tradition of Brouwer
and Heyting, they would have to sacrifice their understanding of mathematical objects as ideal entities existing
independently of one's knowledge about them.
But most of them certainly do not want to be ``expelled from the paradise that Cantor has created'', and for
a good reason: the customary mental aid of Platonism does simplify their job immensely.

(ii) Kolmogorov's problem interpretation of intuitionistic logic entirely avoids the issue of sacrificing
platonist thinking.
But then the double negation translation makes no sense, because, when understood in these terms,
it conflates problems with
propositions; and when corrected so as to respect their distinction, it is no longer
a translation into plain intuitionistic logic.
This ``corrected'' double negation translation (see \S\ref{diamond}) is, actually, quite meaningful from
the viewpoint of
mathematical practice; for instance, $\exists x\,P(x)$, ``there exists an $x$ such that $P(x)$'' is
interpreted as $\neg\wn\neg\exists x\,\oc P(x)$, ``it is impossible to derive a contradiction from
a construction of an $x$ along with a proof of $P(x)$''.
The ``corrected'' double negation translation is essentially
equivalent to Fitting's translation of classical logic into the modal logic QS4.

(iii)
Even though the ``corrected'' double negation translation is no longer a translation into
plain intuitionistic logic, one might still ask if its effect is significant from the viewpoint of
mathematical practice.
The assertion that its effect is trivial is equivalent (see \cite{M2}*{\ref{g2:nabla-K}}) to the so-called
K-principle, $\neg\oc P\to\oc\neg P$, an independent principle of the joint logic of problems and 
propositions.
But the effect of the K-principle is drastic: it immediately rules out independent statements
(see \cite{M0}*{\S\ref{int:BHK-to}} and \cite{M2}*{\S\ref{g2:H-K}}).

\subsection{Related work} \label{related work}

Modern literature contains a number of attempts to blend classical and intuitionistic logics.
On the one hand, there are the Linear Logic and the logics of Japaridze \cite{Ja0},
\cite{Ja1}, \cite{Ja2}, \cite{Ja3} and Liang--Miller \cite{LM1}, \cite{LM2}, which all have
something classical and something intuitionistic in them --- albeit fused in far more elaborate ways
than Kolmogorov could have possibly meant in his words: ``A unified logical apparatus was intended
to be created, which would deal with objects of two types --- propositions and problems'' \cite{Kol3}.%
\footnote{This is a literal translation; the meaning of ``Предполагалось создание ...'' is inherently
ambiguous, and could well be either ``I intended to create ...'' or ``We intended to create with
my colleagues ...'' or ``I intended a student to create ...''.}

On the other hand, there is Art\"emov's Logic of Proofs LP, which he actually meant to address
these very words of Kolmogorov \cite{Ar1}*{p.\ 2}.
It can be said to deal with objects of two types --- propositions and their {\it proofs}; so, it does not 
exactly fit Kolmogorov's description.
In a paper in progress the author studies a proof-relevant extension of the joint logic of problems
and propositions which includes a variation of Art\"emov's LP.

What is more obviously related to Kolmogorov's research program is the ``propositions-as-{\it some}-types''
paradigm, and indeed our composite operator $\oc\wn$ on problems is very similar to the squash/bracket
operator in intuitionistic type theories (see \S\ref{Galois subsection} and \ref{Russell-Prawitz}).
There are also similarities between our approach and some ideas behind the Calculus of Constructions
\cite{Coq} (see also \cite{AG} and \cite{BL}).

A direct type-theoretic analogue of our $\wn$ due to Aczel and Gambino \cite{AG}*{\S1.3} is dissimilar to
$\wn$ in that it satisfies a reversible analogue of our schema ($\wnto$) (see \S\ref{deductive}).
In contrast, the reversibility of our ($\wnto$) would amount to allowing the BHK interpretation
to represent arbitrary, and not just constructive functions (see \cite{M0}*{\S\ref{int:confusion}}).
But there is nothing surprising here, since Aczel and Gambino do not assume the principle of excluded middle
on either the two sides.

A type-theoretic analogue of our $\oc$ due to Coquand \cite{Coq}*{\S1} satisfies an analogue of our
schema ($\oc_\forall$) (see \S\ref{symmetry}), which Coquand argues to express ``Heyting's semantics of
the universal quantification''.
This time we see a full agreement on the syntactic level; but it is remarkable that our formalization of
the BHK clause for the universal quantification is not ($\oc_\forall$), which is reversible just like its
Coquand's version, but ($\wn_\forall$) (see \S\ref{deductive}), which is irreversible for the same
reasons as ($\wnto$).

\section{QHC calculus} \label{QHC}

In the present series of papers we work in first-order logic, but with some deviations from standard terminology,
notation and conventions.
Namely, our basic syntactic setup is the meta-logic of \cite{M0}*{\S\ref{int:formal}}, which is a slightly 
simplified and ``mathematicized'' version of the meta-logic used in the {\tt Isabelle} proof-checker.
The simplification is mostly concerned with omission of features that are not needed for dealing with
first-order logics (without equality).
To be precise, in the present series of papers we use the straightforward extension of the setup in 
\cite{M0}*{\S\ref{int:formal}} to the case of first-order logic with many-sorted predicate variables.

The following includes a quick summary of \cite{M0}*{\S\ref{int:formal}} which should suffice for the reader who 
is familiar with some conventional treatments of first-order logic as well as simply-typed $\lambda$-calculus 
and natural deduction.

\subsection{Simply-typed $\lambda$-calculus} \label{lambda-calculus}
The language in which our logic and its meta-logic are formulated is the simply-typed $\lambda$-calculus with
(binary) products, with \cite{M0} and the present series of papers taking the following deviations from
standard terminology, notation and conventions.

\begin{itemize}

\item The word ``term'' is used in the sense of first order logic, and consequently we speak of
{\it $\lambda$-expressions} to refer to terms in the sense of $\lambda$-calculus.
The word ``arity'' is used is the traditional sense (of logic and mathematics), and consequently we speak
of {\it types} (rather than arities) of $\lambda$-expressions.
The word ``closed'' (as e.g.\ in ``closed formula'') is used in the sense of first-order logic, so we refer
to $\lambda$-expressions that are closed in the sense of $\lambda$-calculus as {\it $\lambda$-closed} ones.

\item Abstraction is written in the style of mathematics, as $x\mapsto T$, and not in the style of logic
and computer science, $\lambda x.T$.
Function application is normally written as $F(T)$ and only in some cases abbreviated as $FT$.
The function type is denoted $\Gamma\too\Delta$; no associativity conventions for $\too$ and $\mapsto$ are
assumed.

\item We omit brackets in iterated products using standard isomorphisms, and use tuples $(T_1,\dots,T_n)$,
also written $\vec T$, which are defined recursively in terms of pairs.
Projection on the $i$th factor of a product is denoted $\pr_i$.
We also use multivariable abstraction $x_1,\dots,x_n\mapsto T$, which is defined recursively in terms of
abstraction and tuples (not just up to $\alpha$-equivalence; see \cite{M0}*{\S\ref{int:simultaneous}}).

\item Substitution is denoted $S|_{x:=T}$ and is undefined whenever some variable is captured.
The same goes for the simultaneous substitution $S|_{\vec x:=\vec T}$.
A $\lambda$-expression of the form $(\vec x\mapsto S)(\vec T)$ may $\beta\eta$-reduce beyond $S|_{\vec x:=\vec T}$;
if such a $\beta\eta$-reduction involves no $\alpha$-conversions, and its result is in $\beta\eta$-normal form,
then this resulting $\lambda$-expression is denoted $S[\vec x/\vec T]$, and the tuple $\vec T$ is called
{\it free for $\vec x$ in S} (see \cite{M0}*{\S\ref{int:free substitution}}).

\item The variables of a type $\Gamma$ are denoted $\var_1^\Gamma,\var_2^\Gamma,\dots$.
We generally use lowercase letters to write metavariables for variables and constants, and uppercase letters
to write metavariables for arbitrary $\lambda$-expressions.
\end{itemize}

\subsection{Language of QHC}

QHC is a first-order logic without equality, whose predicate variables are of two sorts.
To describe its language, we need three basic types:

\begin{itemize}
\item $\0$, the type of terms;
\item $\1_i$, the type of i-formulas (``i'' stands for ``intuitionistic'');
\item $\1_c$, the type of c-formulas (``c'' stands for ``classical'').
\end{itemize}

The language of QHC consists of the following sets of typed $\lambda$-expressions
(variables and constants only), where $n$ ranges over $\N=\{0,1,2,\dots\}$:

\begin{itemize}
\item[(1)] the set of variables of type $\0$, called {\it individual variables};
\item[(2$_n$)] the set of variables of type $\0^n\too\1_c$, called {\it $n$-ary predicate variables};
\item[(3$_n$)] the set of variables of type $\0^n\too\1_i$, called {\it $n$-ary problem variables}.
\end{itemize}

Each of the sets (1), (2$_n$), (3$_n$) is a countably infinite set.
Nullary predicate variables are also called {\it propositional variables}.
For reasons of readability we will also use the alternative spelling $\tr{a,b,c,\dots,x,y,z}$ for the first
26 individual variables $\var_1^\0,\dots,\var_{26}^\0$, reserving an upright sans-serif font for this purpose.
Similarly, we use the abbreviations $\fm {a,b,c,\dots,x,y,z}$ for the first 26 predicate variables of each arity
and $\fm {\alpha,\beta,\gamma,\dots,\chi,\psi,\omega}$ for the first 24 problem variables of each arity,
reserving a fancy (Euler) upright serif font for this purpose.

In using predicate and problem variables we follow the tradition of classic texts in first-order logic such as
those by Hilbert--Ackermann, Hilbert--Bernays, Church and P. S. Novikov, who did include predicate variables
in addition to predicate constants.
Modern treatments of first-order logic usually do not include predicate variables in the language, and are content
with predicate constants (even though they include propositional variables in the language of propositional logic).
In fact, it is clear that the language of a logic in reality contains only predicate variables, whereas
predicate constants are chosen differently for each theory over the logic, and so actually belong to
the language of a theory and not to the language of the logic.

The sets (1)--(3$_n$) are common to any two-sorted first-order logic.
Specific to QHC are the following constants.
{\it Connectives}:

\begin{itemize}
\item[(4)] truth and falsity $\cltop,\clbot:\1_c$;
\item[(5)] classical negation $\neg:\1_c\too\1_c$;
\item[(6)] classical binary connectives $\land,\lor,\to,\tofrom:\1_c\x\1_c\too\1_c$;
\item[(7)] triviality and absurdity $\triv,\ab:\1_i$;
\item[(8)] intuitionistic negation $\neg:\1_i\too\1_i$;
\item[(9)] intuitionistic binary connectives $\land,\lor,\to,\tofrom:\1_i\x\1_i\too\1_i$,
\end{itemize}

\noindent
{\it quantifiers}:

\begin{itemize}
\item[(10)] classical quantifiers $\forall,\exists:(\0\too\1_c)\too\1_c$;
\item[(11)] intuitionistic quantifiers $\forall,\exists:(\0\too\1_i)\too\1_i$,
\end{itemize}

\noindent
and {\it conversion operators}:

\begin{itemize}
\item[(12)] $\oc:\1_c\too\1_i$;
\item[(13)] $\wn:\1_i\too\1_c$.
\end{itemize}

Some of the connectives and quantifiers are ``syntactic sugar'', i.e.\ they should not really be on
the above list as they are definable in terms of others.
Namely, the intuitionistic $\tofrom$, $\neg$ and $\triv$ are definable in terms of the intuitionistic
$\land,\lor,\to$ and $\ab$; and the classical $\tofrom$, $\land$, $\lor$, $\neg$ and $\cltop$ are
definable in terms of the classical $\to$ and $\clbot$, and the classical $\exists$ is definable in terms
of the classical $\forall$, $\to$ and $\clbot$.
However it is convenient to regard all these symbols (4)--(11), including the redundant ones,
as ``connectives'' and ``quantifiers''.

It should be noted that we do not differentiate graphically between classical connectives/quantifiers
and intuitionistic ones, since they can be distinguished by the type of the $\lambda$-expressions
that they act upon ($\1_c$ or $\1_i$) --- except for the nullary connectives, which we do take care
to differentiate (classical: $\cltop,\clbot$; intuitionistic: $\triv,\ab$).
This is based on the observation that lowercase Greek letters, which we use to denote problem variables,
are visually distinct from lowercase Roman letters, which we use to denote predicate variables.
Note, however, the difference between $\to$ (classical or intuitionistic implication) and $\too$ (function type).

If $\frak q$ is a quantifier, $A$ is a $\lambda$-expression of type $\1_c$ or $\1_i$, and $x$ is
an individual variable, then $\frak qx\, A$ abbreviates the $\lambda$-expression $\frak q(x\mapsto A)$.
More generally, $\frak q\vec x\, A$ abbreviates $\frak q(\vec x\mapsto A)$.
Due to this abbreviation, $\lambda$-abstraction is only implicit in formulas.

\begin{remark} In the preceding paragraph, $A$ is a metavariable that stands for an arbitrary unknown
$\lambda$-expression of type $\1_c$ or $\1_i$.
Accordingly, the symbol ``$A$'' can be read in two ways: as the first Roman uppercase letter or as the first 
Greek uppercase letter.
We will use uppercase letters that are unambiguously Greek (from the viewpoint of \TeX) to write metavariables 
that stand unambiguously for a $\lambda$-expression of type $\1_i$, and those unambiguously Roman for
$\lambda$-expressions of type $\1_c$.    
\end{remark}

This completes the description of the pure language of QHC.
However, the language $\L$ of a theory over QHC (such as the plane geometry of \cite{M3}) may additionally contain
the following sets:

\begin{itemize}
\item[(14$_n$)] a finite set of constants of type $\0^n\too\0$, called {\it $n$-ary function symbols};
\item[(15$_n$)] a finite set of constants of type $\0^n\too\1_c$, called {\it $n$-ary predicate constants};
\item[(16$_n$)] a finite set of constants of type $\0^n\too\1_i$, called {\it $n$-ary problem constants}.
\end{itemize}

It should be noted that nullary predicate and problem constants are the same kind of $\lambda$-expressions as
nullary connectives (i.e., constants of types $\1_c$ and $\1_i$).
It is nevertheless convenient to distinguish them, since the latter belong to the pure language of QHC but
the former do not.

{\it Terms} of the language $\L$ are defined inductively, as built out of individual variables using
the function symbols.
Thus not every $\lambda$-expression of type $\0$ is a term (for example, no term involves $\lambda$-abstraction).
An {\it atomic c-formula} of $\L$ is a $\lambda$-expression of type $\1_c$ obtained by applying either
an $n$-ary predicate constant or an $n$-ary predicate variable to an $n$-tuple of terms; an {\it atomic
i-formula} is a $\lambda$-expression of type $\1_i$ obtained by applying either an $n$-ary problem constant or
an $n$-ary problem variable to an $n$-tuple of terms.
A {\it formula} of $\L$ is a $\lambda$-expression built out of atomic c-formulas and i-formulas
using the connectives, quantifiers and conversion operators.

A formula of type $\1_c$ is called a {\it c-formula} and a formula of type $\1_i$ is called an {\it i-formula}.
(Clearly, every formula is either a c-formula or an i-formula.)

A {\it purely classical formula} is a $\lambda$-expression of type $\1_c$ built out of atomic c-formulas using
classical connectives and classical quantifiers only; a {\it purely intuitionistic formula} is 
a $\lambda$-expression of type $\1_i$ built out of atomic i-formulas using intuitionistic connectives 
and intuitionistic quantifiers only.

A $\lambda$-expression of the form $x_1,\dots,x_n\mapsto F$, where $F$ is a formula and $x_1,\dots,x_n$ are
pairwise distinct individual variables, is called an {\it $n$-formula}.
It can also be called an {\it $n$-c-formula} or an {\it $n$-i-formula} if $F$ is a c-formula or an i-formula.

\subsection{Meta-logic}  \label{meta-logic-section}

\subsubsection{Introduction}
Any kind of literature on first-order logic constantly deals with meta-logical concepts and assertions,
but usually only implicitly.
Why would one want to make them explicit, and discuss a first-order logic in terms of a formal meta-logic?
One reason is that a pedantic verbalist, who ignores the implicit, must perceive the hidden meta-logic
as an ever-present conflation and ambiguity.
Here are two examples.

\begin{example}
The literature on first-order classical and intuitionistic logics is accustomed to speaking of
``the syntactic consequence''; but the syntactic consequence in the sense of e.g.\ the textbooks by Schoefield 
and Mendelson is inequivalent to the syntactic consequence in the sense of e.g.\ the textbooks by Church, 
Enderton, Kolmogorov--Dragalin and Troelstra--van Dalen.
Moreover, Kleene and Avron have considered the two notions simultaneously, as well as the corresponding
notions of semantic consequence, pointing out that both are commonly used in elementary mathematics.

Kleene's textbook contains the following example: the arithmetical formula $(x+y)^2=x^2+2xy+y^2$ begs to be
understood as an identity (valid for all natural numbers $x$), whereas the arithmetical formula $x^2+2=3x$
begs to be understood as an equation (i.e., as a condition on $x$).
There is no special syntax to reflect this obvious distinction in meaning.
Yet it is not illusory, as it is reflected in use.
For, as noted by Avron, when ``dealing with identities [...] the substitution rule is available,
and one may infer $\sin x=2\sin\frac x2\cos\frac x2$ from the identity $\sin 2x=2\sin x\cos x$.
In contrast, [...] substituting $\frac x2$ for $x$ everywhere in an equation is an error'' (see
references in \cite{M0}*{\S\ref{int:formal}}).

In fact, the difference between the two variants of syntactic consequence is due to the implicit presence
of a first-order meta-quantifier in one of them.
\end{example}

\begin{example} \label{principles vs schemata}
In intuitionistic logic, the principle of excluded middle is derivable from the double negation principle
(due to the derivability of the schema $\neg\neg(\gamma\lor\neg\gamma)$).
Nevertheless, the schema $\alpha\lor\neg\alpha$ expressing the principle of excluded middle is not derivable
from the schema $\neg\neg\alpha\to\alpha$ expressing the double negation principle (since for $\alpha=\neg\beta$
the latter is derivable, and the former is not).
Thus the widespread practice of expressing principles by schemata is sometimes misleading.

In fact, the difference between principles and schemata is due to the implicit presence of a second-order
meta-quantifier in principles.

But, actually, the explicit use of the second-order meta-quantifier makes the whole concept of schemata
(i.e., the formal use of metavariables for this purpose) superfluous.
Let us recall that early textbooks on first-order logic, such as those of Hilbert--Ackermann, Hilbert--Bernays
and P. S. Novikov did not speak of any schemata, but only of formulas; instead, their derivation systems
included a substitution rule.
Some problems with this early approach are that inference rules were anyway stated in schematic form,
and also that the substitution rule is, in contrast to other inference rules, not structural (i.e.\ it is not
preserved itself by substitution without anonymous variables).
Non-structurality is a serious complication in trying to treat rules as fully formal objects.

In fact, the use of both first-order and second-order meta-quantifiers enables one to state (structural) rules
without using meta-variables; and one way to understand the substitution rule is that it is not an inference rule
of the logic, but an inference meta-rule of the meta-logic.
An advantage of this approach is that side conditions that normally occur in first-order logics, such as
``provided that $x$ is not free in $\alpha$'' or ``provided that $t$ is free for $x$ in $\alpha(x)$'' effectively
disappear (more precisely, they remain at the meta-level, but they disappear from what needs to be specified in
order to state rules and principles).
One consequence of not having to specify exactly which English phrases qualify as ``side conditions'' in rules
and principles is that it becomes feasible to give actual formal definitions of these notions 
(a rule and a principle) as well as further notions such as a derivable rule, an admissible rule, a first-order 
logic, and (both variants of) syntactic consequence.
\end{example}

\subsubsection{Meta-formulas} \label{meta-logic-language}
The language of the meta-logic%
\footnote{Not to be confused with the meta-language of a logic.
(This one would have to be formalized if we were to give a formal treatment of schemata.)}
of a two-sorted first-order logic involves, in addition to the basic types
$\0$, $\1_i$ and $\1_c$, a fourth basic type:

\begin{itemize}
\item $\m$, the type of meta-formulas;
\end{itemize}

\noindent
and consists of the following constants (common to all two-sorted first-order logics).
{\it Reflection operators}:

\begin{itemize}
\item $\Rt_i:\1_i\too\m$, the i-reflection;
\item $\Rt_c:\1_c\too\m$, the c-reflection,
\end{itemize}

\noindent
{\it meta-connectives}:

\begin{itemize}
\item $\mand:\m\x\m\too\m$, the meta-conjunction;
\item $\imp:\m\x\m\too\m$, the meta-implication,
\end{itemize}

\noindent
and {\it meta-quantifiers}

\begin{itemize}
\item $\q:(\0\too\m)\too\m$, the first-order (universal) meta-quantifier;
\item $\q^n_i:((\0^n\too\1_i)\too\m)\too\m$, the $n$-ary second-order (universal) i-meta-quantifier;
\item $\q^n_c:((\0^n\too\1_c)\too\m)\too\m$, the $n$-ary second-order (universal) c-meta-quantifier.
\end{itemize}

Here $n$ ranges over $\N=\{0,1,2,\dots\}$.
In practice, meta-quantifiers are written like the old-style (early 20th century) universal quantifiers, but
with fancy parentheses $\mq{\cdot}$ so as to avoid visual confusion with the ordinary parentheses $(\cdot)$:
if $\frak q:(\Delta\too\m)\too\m$ is a meta-quantifier (either of them), $\F$ is a $\lambda$-expression of
type $\m$, and $x$ is a variable of type $\Delta$, then $\mq{x}\,\F$ abbreviates the $\lambda$-expression
$\frak q(x\mapsto \F)$.
More generally, $\mq{x_0,\dots,x_n}\,\F$ abbreviates $\mq{x_0}(\mq{x_1,\dots,x_n}\,\F)$.

An {\it atomic meta-formula} is a $\lambda$-expression of type $\m$ that is either of the form $\Rt_c F$,
where $F$ is a c-formula, or of the form $\Rt_i\Phi$, where $\Phi$ is an i-formula.
A {\it meta-formula} is a $\lambda$-expression of type $\m$ built out of atomic meta-formulas using
meta-connectives and meta-quantifiers.
We usually omit $\Rt_c$ and $\Rt_i$ in writing $\lambda$-expressions of type $\m$;
thus atomic meta-formulas are effectively identified with formulas, keeping in mind that meta-connectives and
meta-quantifiers cannot be used inside of formulas.

As usual, $\F\iff\G$ abbreviates $(\F\imp\G)\mand(\G\imp\F)$; ``$\iff$'' is called {\it meta-equivalence}.
We stick to the following order of precedence of logical and meta-logical symbols (in groups of equal
priority, starting with higher precedence/stronger binding):
\begin{enumerate}
\item $\oc$, $\wn$, $\neg$, $\exists$ and $\forall$;
\item $\land$ and $\lor$;
\item $\to$ and $\tofrom$;
\item $\mq{\cdot}$;
\item $\&$;
\item $\imp$ and $\iff$.
\end{enumerate}

\subsubsection{Meta-rules} \label{meta-logic-rules}
The inference meta-rules (i.e., the inference rules of the meta-logic) are the $\alpha$-conversion rule
for meta-formulas:

\begin{center}
$\Dfrac{\begin{matrix}\vdots\\ \F\end{matrix}}{\G}$, if $\F$ is $\alpha$-equivalent to $\G$,
\end{center}
\bigskip

\noindent
and the usual introduction/elimination rules of natural deduction for $\mand$, $\imp$ and the meta-quantifiers:

\begin{center}
$\Dfrac{\begin{matrix}\vdots\\ \F\end{matrix}\qquad
\begin{matrix}\vdots\\ \G\end{matrix}}{\F\mand\G}\qquad\qquad
\Dfrac{\begin{matrix}\vdots\\ \F\mand\G\end{matrix}}{\F}\qquad\qquad
\Dfrac{\begin{matrix}\vdots\\ \F\mand\G\end{matrix}}{\G}\qquad\qquad
\Dfrac{\begin{matrix}\vdots\\ \F\end{matrix}\qquad
\begin{matrix}\vdots\\ \F\imp\G\end{matrix}}{\G}\qquad\qquad
\Dfrac{\begin{matrix}[\F]\\ \vdots\\ \G\end{matrix}}{\F\imp\G}$,
\end{center}
\bigskip

\noindent
where $\F$ and $\G$ are meta-formulas;

\begin{center}
$\Dfrac{\begin{matrix}\vdots\\ \F\end{matrix}}{\mq{x}\F}$,
provided that $x$ does not occur freely in any of the assumptions;
\end{center}
\begin{center}
$\Dfrac{\begin{matrix}\vdots\\ \mq{x}\F\end{matrix}}{\F[x/T]}$,
provided that $T$ is free for $x$ in $\F$,
\end{center}
\bigskip

\noindent
where $\F$ is a meta-formula, and there are three ways to read $x$ and $T$:

\begin{enumerate}
\item $x$ is an individual variable and $T$ is a term;
\item $x$ is an $n$-ary problem variable and $T$ is an $n$-i-formula;
\item $x$ is an $n$-ary predicate variable, $T$ is an $n$-c-formula.
\end{enumerate}

It should be noted that $\F[x/T]$ boils down to the ordinary substitution $\F|_{x:=T}$ of $\lambda$-calculus
in the case (1), but not in the cases (2), (3) (see \S\ref{lambda-calculus} above).

Let us note that by using a {\it meta-specialization} (=meta-quantifier elimination meta-rule) immediately after
the corresponding {\it meta-generalization} (=meta-quantifier introduction meta-rule), we get the meta-rules
of substitution:

\begin{center}
$\Dfrac{\begin{matrix}\vdots\\ \F\end{matrix}}{\F[x/T]}$,
as long as $x$ does not occur freely in the assumptions and $T$ is free for $x$ in $\F$.
\end{center}
\bigskip

A meta-formula $\F$ is called {\it deducible} if using the meta-rules one can obtain (from the trivial deductions,
in which a meta-formula is deduced from itself) a deduction of $\F$ from no assumptions.

\subsubsection{Syntactic meta-sugar}

The {\it first-order meta-closure} $\mc{1}\F$ of the meta-formula $\F$ is $\mq{\vec x}\F$,
where $\vec x$ is the tuple of all individual variables occurring freely in $\F$.
The {\it second-order meta-closure} $\mc{2}\F$ is $\mq{\vec\gamma}\F$,
where $\vec\gamma$ is the tuple of all predicate and problem variables occurring freely in $\F$.

A {\it rule}, written $ A_1,\dots, A_m/ B$, or, in more detail,
\[\frac{ A_1,\dots, A_m}{ B},\]
where $A_1,\dots,A_m$ and $B$ are formulas, is an abbreviation for the meta-formula
\[\mc{2}\Big(\mc{1} A_1\,\mand\dots\mand\,\mc{1} A_m\Imp \mc{1} B\Big).\]
The formulas $ A_1,\dots, A_m$ are called the {\it premisses} of the rule, and $ B$ its
{\it conclusion}.

If $B$ is a formula (and only in this case) we abbreviate $\mc{2}\mc{1} B$ by $\prin B$.
A meta-formula of the form $\prin B$, where $B$ is a formula, is called a {\it principle}.
In other words, a principle is a formula that is meta-quantified over all its free (individual,
predicate and problem) variables.
Rules with no premisses can be identified with principles, in the sense that each meta-formula of the form
$(\,/ B)\iff\prin B$ is deducible, as long as the empty meta-conjunction is defined as an abbreviation
of some deducible meta-formula (for example, $\mq\gamma\gamma\imp\mq\gamma\gamma$).

The difference between formulas and principles is clear from Example \ref{principles vs schemata}:
in (the meta-logical extension of) intuitionistic logic, the meta-formula
\[\fm{\prin\neg\neg\alpha\to\alpha\imp\prin\alpha\lor\neg\alpha}\]
is deducible, whereas the meta-formula
\[\fm{\neg\neg\alpha\to\alpha\imp\alpha\lor\neg\alpha}\]
is not deducible.

A {\it derivation system}%
\footnote{Also called a ``deductive system'' in the literature.
For our purposes it is convenient to distinguish meta-logical {\it deductions} from {\it derivations}
in a specific logic.}
$\Ds$ is a meta-formula of the form \[\H_1\mand\dots\mand\H_k,\]
where each $\H_i$ is a rule (possibly with no premisses) in the pure language of QHC.
The $\H_i$ with no premisses, or rather the corresponding principles, are called the {\it laws},
and the $\H_i$ with at least one premise are called the {\it inference rules}.

A {\it logic} is a meta-equivalence class of derivation systems.
In other words, derivation systems $\Ds$ and $\Ds'$ are said to {\it determine the same logic} if
the meta-formula $\Ds\iff\Ds'$ is deducible.

A meta-formula $\F$ is called {\it derivable} in the logic determined by a derivation system $\Ds$
if the meta-formula $\Ds\imp\F$ is deducible (in the meta-logic).
Clearly, adding a derivable principle or rule to a derivation system $\Ds$
does not affect derivability of principles and rules in the logic determined by $\Ds$.

If $L$ is the logic determined by a derivation system $\Ds$, we denote by $\turnstile\F$,
or in more detail $\turnstile _L\F$, the judgement that the meta-formula $\F$ is derivable in the logic.
The meta-meta-logical symbol $\turnstile$ is set to have lower priority than all logical and meta-logical
symbols.
The judgement $\turnstile \F_1\mand\dots\mand\F_m\imp\G$ is also abbreviated by
\[\F_1,\dots,\F_m\turnstile\G.\]
When this judgement is true, we also say that $\G$ is a {\it (syntactic) consequence} of the $\F_i$.
This yields two notions of syntactic consequence for formulas:
$ A_1,\dots, A_m\turnstile B$ is the traditional ``fixed variables'' one, as in the textbooks by
Church, Troelstra and van Dalen; whereas $\mc{1} A_1,\dots,\mc{1} A_m\turnstile \mc{1} B$
is the traditional ``varied variables'' one, as in the textbooks by Schoenfield and Mendelson.
There seems to be no standard notation for the judgement of interderivability for formulas:
$$\turnstile A\iff B$$
so we will keep it in this form. Let us note that, due to the absence of the deduction theorem in QHC, 
it is weaker than the (object-level) equivalence (which makes sense when both if $A$ and $B$ are either  
i-formulas or c-formulas),
$$\turnstile A\tofrom B,$$
but stronger than the equivalence of principles,
$$\turnstile\prin A\iff\prin B,$$
which is in turn stronger than the equivalence of judgements:
$$\turnstile A\text{ if and only if }\turnstile B.$$

\subsection{Derivation system}\label{deductive}

When writing down a derivation system for a new logic, one has to engage in informal considerations,
or else risk the new logic being entirely unmotivated.

To provide an informal mathematical meaning to the judgements of QHC, we interpret c-formulas
by propositions/predicates and i-formulas by problems.
More precisely, we instantiate predicate variables and problem variables by particular mathematical predicates
and problems.
Upon such instantiation, classical connectives and quantifiers are interpreted according to the usual truth tables;
intuitionistic connectives and quantifiers according to the BHK interpretation, in Kolmogorov's problem solving 
terminology (see below); and the conversion operators $\oc$ and $\wn$ are interpreted as in \S\ref{intro}.
The interpretation of the meta-logical constants and judgements will be discussed in part II.

Some laws and inference rules of the QHC calculus are immediate:

\begin{itemize}
\item All laws and inference rules of classical predicate logic (see \cite{M0}*{\S\ref{int:logics}}).
\item All laws and inference rules of intuitionistic first-order logic (see \cite{M0}*{\S\ref{int:logics}}).
\end{itemize}

Let us note that by using the substitution (i.e., meta-generalization followed by meta-specialization)
we can apply the classical laws and inference rules to arbitrary i-formulas (possibly involving $\wn$ and $\oc$)
and the intuitionistic laws and inference rules to all c-formulas (possibly involving $\wn$ and $\oc$).

We will now discuss the remaining part of the derivation system.

\subsubsection{From the problem interpretation} \label{deductive1}
Let us recall Kolmogorov's problem interpretation of intuitionistic logic \cite{Kol} (with minor improvements 
largely due to Heyting; see \cite{M0}*{\S\ref{int:BHK}, \S\ref{int:about-bhk}} for further details).%
\footnote{This can also be understood as the BHK interpretation presented in Kolmogorov's language.
However, given that Heyting's early ideas are often conflated with the BHK interpretation in the literature,
but will be understood in a very different way below, as providing a {\it complement} to the BHK interpretation, 
one must be very careful here about exactly what is meant by the ``BHK interpetation''.}

We fix a prescribed class of specific problems, which may have parameters that run over a fixed domain $D$.
These are our {\it primitive} problems, and we assume that it is known what is a solution of each primitive problem 
for each value of the parameters.
For instance, Euclid's first three postulates are the following primitive problems:

(1) draw a straight line segment from a given point to a given point;

(2) extend any given straight line segment continuously to a longer one;

(3) draw a circle with a given center and a given radius.

\noindent
We may thus stipulate that each of (1) and (3) has a unique solution, and describe all possible solutions of (2).
(Euclid's {\it Elements} is discussed in some detail in part III of the present series, \cite{M3}.)

Composite problems are obtained from the primitive ones by using connectives $\land$, $\lor$, $\to$, 
$\neg$, $\ab$ and quantifiers $\forall$, $\exists$.
What it is a solution of a composite problem is explained as follows:

\begin{itemize}
\item a solution of $\Gamma\land\Delta$ consists of a solution
of $\Gamma$ and a solution of $\Delta$;

\item a solution of $\Gamma\lor\Delta$ consists of an explicit choice
between $\Gamma$ and $\Delta$ along with a solution of the chosen
problem;

\item a solution of $\Gamma\to\Delta$ is a {\it reduction} of $\Delta$ to
$\Gamma$; that is, a general method of solving $\Delta$ on the basis
of any given solution of $\Gamma$;

\item the {\it absurdity} $\ab$ has no solutions;
$\neg\Gamma$ is an abbreviation for $\Gamma\to\ab$;

\item a solution of $\exists x\, \Theta(x)$ is a solution of $\Theta(x_0)$
for some explicitly chosen $x_0\in D$;

\item a solution of $\forall x\, \Theta(x)$ is a general method
of solving $\Theta(x_0)$ for all $x_0\in D$.
\end{itemize}

A key element here is the notion of a {\it general method} (roughly corresponding to the notion of ``construction'' 
advocated by Brouwer and Heyting), which Kolmogorov further explains as follows.
If $\Gamma(\script X)$ is a problem depending on the parameter $\script X$ ``of any sort'', then 
``to present a general method of solving $\Gamma(\script X)$ for every particular value of $\script X$''
should be understood as ``to be able to solve $\Gamma(\script X_0)$ for every given specific value of 
$\script X_0$ of the variable $\script X$ by a finite sequence of steps, known in advance (i.e.\ before 
the choice of $\script X_0$)''.

Let us observe that if $|\Gamma|$ denotes the set of solutions of the problem $\Gamma$, then the above clauses 
guarantee that:
\begin{itemize}
\item $|\Gamma\land\Delta|$ is the product $|\Gamma|\x|\Delta|$;
\item $|\Gamma\lor\Delta|$ is the disjoint union $|\Gamma|\sqcup|\Delta|$;
\item there is a map $\script F\:|\Gamma\to\Delta|\to\Hom(|\Gamma|,|\Delta|)$ into the set of all maps;
\item $|\ab|=\emptyset$;
\item $|\exists x\,\Theta(x)|$ is the disjoint union $\bigsqcup_{d\in D} |\Theta(d)|$;
\item there is a map $\script G\:|\forall x\,\Theta(x)|\to\prod_{d\in D} |\Theta(d)|$ into the product.
\end{itemize}

Now the proposition ``$\Gamma$ has a solution'' can be rephrased as ``$|\Gamma|\neq\emptyset$''.
It follows that the following propositions must be true for any contentful problems $\Gamma$, $\Delta$ 
and any contentful parametric problem $\Theta$:

\begin{itemize}
\item $\wn(\Gamma\land\Delta)\Tofrom\wn\Gamma\land\wn\Delta$;
\item $\wn(\Gamma\lor\Delta)\Tofrom\wn\Gamma\lor\wn\Delta$;
\item $\wn(\Gamma\to\Delta)\To(\wn\Gamma\to\wn\Delta)$;
\item $\neg\wn\ab$;
\item $\wn\exists x\,\Theta(x)\Tofrom\exists x\,\wn\Theta(x)$;
\item $\wn\forall x\,\Theta(x)\To\forall x\,\wn\Theta(x)$.
\end{itemize}

See \cite{M0}*{\S\ref{int:about-bhk}} for a more thorough discussion of these propositions.

This motivates some laws of QHC (beware that some of these will turn out to be redundant):

\formulas
\medskip
\begin{enumerate}
\item[($\wn_\land$)] $\prin\wn(\gamma\land\delta)\Tofrom \wn\gamma\land\wn\delta$;
\item[($\wn_\lor$)] $\prin\wn(\gamma\lor\delta)\Tofrom \wn\gamma\lor\wn\delta$;
\item[($\wnto$)] $\prin\wn(\gamma\to \delta)\To (\wn\gamma\to\wn\delta)$;
\item[($\wn_\bot$)] $\neg\wn\ab$;
\item[($\wn_\exists$)] $\prin\wn\exists\tr x\,\theta(\tr x)\Tofrom \exists\tr x\,\wn\theta(\tr x)$;
\item[($\wn_\forall$)] $\prin\wn\forall\tr x\,\theta(\tr x)\To\forall\tr x\,\wn\theta(\tr x)$.
\end{enumerate}
\medskip
\metameta

It should be noted that formulas with almost same appearance and motivation, but somewhat different meaning
appear in \cite{M0}*{\S\ref{int:weak BHK}}.

Informally, ($\wn_\bot$) is saying that $\ab$ is not just the hardest problem (as guaranteed by the
explosion principle, $\fm{\prin\ab\to\gamma}$), but a problem that has no solutions whatsoever.
This is just the first example of how some content found in the BHK interpretation and not entirely
captured in the usual formalization of intuitionistic logic is more fully captured in QHC.

Some versions of the BHK interpretation include the well-known principle (see \cite{M0}*{\S\ref{int:about-bhk}}),
that every solution of a problem $\Gamma$ must be supplied with a proof that is it indeed a solution of 
$\Gamma$.
This principle was emphasized by G. Kreisel in connection with interpreting intuitionistic logic (in a somewhat
different form) and also by the ancient Greeks, particularly Proclus, in the context of geometric construction
problems, which as we now know can be seen as a model of intuitionistic logic (see \cite{M3}).
This Proclus--Kreisel principle is usually considered to be relevant when one tries to make sense out the BHK
interpretation in the context of first-order logic, rather than a constructive type theory (see references in
\cite{M0}*{\S\ref{int:about-bhk}}).

A consequence of this Proclus--Kreisel principle is that a solution of a problem $\Gamma$ yields a proof of
the existence of a solution of $\Gamma$.
This is expressible in the language of QHC:

\begin{enumerate}
\smallskip
\item[($\oc\wn$)] $\fm{\prin\gamma\to \oc\wn\gamma}$.
\end{enumerate}
\smallskip

\subsubsection{From the proof interpretation}\label{Goedel}
The remaining part of the derivation system is motivated by the proof interpretation of intuitionistic logic,
given independently by Orlov and Heyting (see details in \S\ref{letters1}) and partially formalized in
G\"odel's translation
of intuitionistic logic into classical modal logic S4 (see \cite{M0}*{\S\ref{int:provability}}).
A remarkable attempt to clarify the informal notion of ``proof'' used by Orlov and Heyting occurs in
G\"odel's sketch of a proof-relevant analogue of S4, which is found in his outline of a 1938 lecture,
published posthumously in his collected works \cite{Goe1}.

G\"odel's proposal is based on a ternary relation ``$zBp,q$, that is, $z$ is a derivation of $q$ from $p$''.
But as a matter of fact he also uses a binary relation ``$aBq$'' which is presumably meant to abbreviate
$aB\cltop,q$.
Here $B$ stands for German {\it Beweis} (proof), and apparently refers to proofs ``understood not in a particular
system, but in the absolute sense (that is, one can make it evident)'' (these words of G\"odel appears earlier on
the same page).
G\"odel's axioms for $B$ are as follows (literally):
\smallskip

\begin{enumerate}
\item ``$zB\phi(x,y)\To \phi(x,y)$'';
\item ``$uBv\To u'B(uBv)$'';
\item ``$zBp,q\mand uBq,r\To f(z,u)Bp,r$'';
\item ``if $q$ has been proved and $a$ is the proof, [then] $aBq$ is to be written down''.
\end{enumerate}
\medskip

Instead of attempting to clarify the meaning of this in G\"odel's original terms, 
let us consider a very similar but more clearly described logic.
Namely, let S4pr be the extension of classical predicate logic with the following 
additional elements of the language:

\begin{itemize}
\item an operator \,$\suchthat$\, associating to every formula $F$ and every term $t$
a formula $t\suchthat F$;
\item a unary function $'$ that associates to every term $t$ a term $t'$;
\item a binary function $[\cdot]$ that associates to every two terms $s,t$ a term $s[t]$;
\item an operator $*$ that associates to every formula $F$ a term $*_F$,
\end{itemize}

\noindent
and the following additional laws and inference rules:

\formulas
\begin{enumerate}[label=(\roman*)]
\item $\prin\tr t\suchthat p\To p$;
\item $\prin\tr t\suchthat p\To\tr t'\suchthat (\tr t\suchthat p)$;
\item $\prin\tr s\suchthat (p\to q)\To (\tr t\suchthat p\to\tr s[\tr t]\suchthat q)$;
\smallskip
\item $\dfrac{p}{*_p\suchthat p}$.
\end{enumerate}
\medskip
\metameta

S. Art\"emov discovered that a further extension of S4pr by an additional function (``sum of proofs'')
and an additional law (not hinted at in any way by G\"odel) is indeed a proof-relevant
analogue of S4 in a sense one could expect \cite{Ar1}.%
\footnote{In fact, the rule $\fm{p/*_p}$ is only applied to axioms in Art\"emov's logic.
The reason why one cannot do without the ``sum of proofs'' is clear from \cite{Ar1}*{Example 5.6}.}
But we do not need these for our motivational purposes.

\formulas
The logic S4pr has the following derived principles and rules:
\begin{enumerate}
\smallskip
\item[(i$'$)] $\neg (\tr t\suchthat\clbot)$;
\item[(i$''$)] $\prin \exists \tr t\ \tr t\suchthat p\To p$;
\smallskip
\item[(i$'''$)] $\dfrac{\tr t\suchthat p}p$;
\smallskip
\item[(ii$'$)] $\prin\tr t\suchthat p\To\tilde{\tr t}\suchthat (\exists\tr t\ \tr t\suchthat p)$.
\end{enumerate}
\smallskip

Here (i$'$) is just the special case of (i) with $p$ substituted by the classical falsity $\clbot$.
Next, (i$''$) is derived from (i) by using two inference rules of classical logic:
$q(t)\,/\,\forall t\, q(t)$ and $\forall t\, (r(t)\to p)\,/\,\exists t\, r(t)\to p$.
Of course, (i$'''$) is derived from (i) using the {\it modus ponens} rule.
To establish (ii$'$), let us first note that from the classical law $q(t)\to\exists t\,q(t)$ we get
$t\suchthat p\to\exists t\ t\suchthat p$, and if $\frak F$ denotes the latter formula, then
by (iv) we get $*_{\frak F}\suchthat(t\suchthat p\to\exists t\ t\suchthat p)$.
Now from (iii) and the {\it modus ponens} rule we get
$t'\suchthat(t\suchthat p)\to *_{\frak F}[t']\suchthat(\exists t\ t\suchthat p)$.
Finally, (ii$'$) follows from this and (ii), if we set $\tilde t=*_{\frak F}[t']$.
\metameta

Just like G\"odel's proofs ``in the absolute sense'', the ``proofs'' of propositions referred to in
the intended reading of the problem $\oc P$, {\it Find a proof of $P$}, are not supposed to be
formal proofs.
In the language of QHC, we have the following direct analogues of (i$'$), (i$''$), (ii$'$), (iii), (iv)
and (i$'''$):

\formulas
\smallskip
\begin{enumerate}
\item[($\oc_\bot$)] $\neg\oc\clbot$;
\item[($\wn\oc$)] $\prin\wn\oc p\to p$;
\item[($\oc\wn'$)] $\prin\oc p\to \oc\wn\oc p$;
\item[($\octo$)] $\prin\oc(p\to q)\To(\oc p\to \oc q)$;
\medskip
\item[($\oc_\top$)] $\dfrac{p}{\oc p}$;
\medskip
\item[($\oc_\top'$)] $\dfrac{\oc p}{p}$.
\medskip
\end{enumerate}
Here ($\oc_\bot$) is a kind of internal soundness: a proof of falsity leads to absurdity.
Semantically (informally), this is pretty much like in G\"odel's system; but let us note that (i$'$)
is a c-formula ($\tr t\suchthat\clbot\to\clbot$), whereas ($\oc_\bot$) is an i-formula
($\oc\clbot\to\ab$).
In contrast, ($\wn_\bot$) is a c-formula ($\wn\ab\to\clbot$).
Note that by the explosion principle, the reverse implications to ($\wn_\bot$) and ($\oc_\bot$) are trivial.
Thus ($\wn_\bot$) identifies the classical falsity, $\clbot$, with the proposition ``$\ab$ has a solution'';
and ($\oc_\bot$) identifies the intuitionistic absurdity, $\ab$, with the problem ``Prove $\clbot$''.

This completes the list of additional inference rules and laws of QHC.
Let us note that ($\oc\wn'$) can be dropped from this list since it follows immediately from
($\wn\oc$), ($\oc_\top$) and ($\octo$).
Some other laws will be shown to be redundant in \ref{symmetry}.
\metameta

\section{Symmetries and redundancy}\label{Galois subsection}

\subsection{Galois connection}

\formulas
\begin{proposition} The inference rule {\rm ($\oc_\top'$)} is equivalent to the following inference rule:
\begin{enumerate}
\item[\rm ($\wn_\top$)] $\dfrac{\gamma}{\wn\gamma}$.
\end{enumerate}
\end{proposition}

We will see in \cite{M2} that the converse rule, $\wn\gamma\,/\,\gamma$, is not derivable in QHC.

\begin{proof} Given ($\oc_\top'$), we can derive ($\wn_\top$) using ($\oc\wn$):
$\gamma,\,\gamma\to\oc\wn\gamma\,/\,\oc\wn\gamma$ and $\oc\wn\gamma\,/\,\wn\gamma$.
Conversely, given ($\wn_\top$), we can derive ($\oc_\top'$) using ($\wn\oc$):
$\oc p\,/\,\wn\oc p$ and $\wn\oc p,\,\wn\oc p\to p\,/\,p$.
\end{proof}
\metameta

The equivalence relations $\turnstile\Phi\tofrom\Psi$ on i-formulas and $\turnstile F\tofrom G$
on c-formulas yield the ``Lindenbaum'' poset of equivalence classes of i-formulas, ordered by 
$[\Phi]\ge[\Psi]$ if $\turnstile\Phi\to\Psi$, and the ``Lindenbaum'' poset of equivalence classes of 
c-formulas, ordered by $[F]\ge[G]$ if $\turnstile F\to G$.
By ($\wn_\top$) and ($\wnto$), and respectively ($\oc_\top$) and ($\octo$) we have:
\begin{itemize}
\item $\turnstile\Phi\to\Psi$ implies $\turnstile\wn\Phi\to \wn\Psi$;
\item $\turnstile F\to G$ implies $\turnstile\oc F\to \oc G$.
\end{itemize}
Thus $\wn$ and $\oc$ descend to monotone maps between the two posets.
Using the monotonicity of $\wn$ and $\oc$ and substitution, from ($\wn\oc$) and ($\oc\wn$) we also obtain:
\begin{itemize}
\item $\turnstile\oc\wn\oc F\tofrom\oc G$;
\item $\turnstile\wn\oc\wn\Phi\tofrom\wn\Psi$.
\end{itemize}
These identities resemble well-known properties of a Galois connection.
Indeed, it turns out that our two monotone maps do form a Galois connection between the two Lindenbaum posets:

\begin{theorem}\label{Galois} For an i-formula $\Phi$ and a c-formula $F$,
$\turnstile\wn\Phi\to F$ if and only if $\turnstile\Phi\to\oc F$.
\end{theorem}

The same argument works to prove a slightly stronger assertion, $\fm{\turnstile\wn\alpha\to p\iff\alpha\to\oc p}$.

\begin{proof} If $\turnstile\Phi\to \oc F$, then $\turnstile\wn\Phi\to \wn\oc F$. 
So from ($\wn\oc$) we get $\turnstile\wn\Phi\to F$.

Conversely, if $\turnstile\wn\Phi\to F$, then $\turnstile\oc\wn\Phi\to\oc F$.
So from ($\oc\wn$) we get $\turnstile\Phi\to\oc F$.
\end{proof}

Another standard fact on Galois connections takes the following form in our situation.

\begin{corollary} \label{sup-inf}
Let $F$ denote a c-formula and let $\Phi$ denote an i-formula.

(a) $[\oc F]$ is the least among all $[\Phi]$ such that $[\wn\Phi]$ is an upper bound of $[F]$;
and $[\wn\Phi]$ is the greatest among all $[F]$ such that $[\oc F]$ is a lower bound of $[F]$.

(b) $[\wn\oc F]$ is the least of all upper bounds of $[F]$ of the form $[\wn\Phi]$; 
and $[\oc\wn\Phi]$ is the greatest of all lower bounds of $[\Phi]$ of the form $[\oc F]$.
\end{corollary}

\begin{proof} The first assertion of (a) says that $\turnstile\wn\oc F\to F$, and if $\turnstile\wn\Phi\to F$, 
then $\turnstile\Phi\to \oc F$.
This is indeed so by ($\wn\oc$) and by \ref{Galois}.
The first assertion of (b) says that $\turnstile\wn\oc F\to F$, and if $\turnstile\wn\Phi\to F$, then 
$\turnstile\wn\Phi\to\wn\oc F$.
This follows similarly, using additionally the monotonicity of $\wn$.
The second assertions of (a) and (b) are proved similarly.
\end{proof}

\subsection{Modalities} \label{modalities}
Let us write $\Box F$ for the c-formula $\wn\oc F$, and $\nabla\Phi$ for the i-formula $\oc\wn\Phi$.
Upon substituting problems and propositions for the atoms of $F$ and $\Phi$, these are interpreted 
by the proposition $\Box P$, ``There exists a proof of $P$'', and the problem $\nabla\Gamma$, 
``Prove that $\Gamma$ has a solution''.
By another standard fact on Galois connections, the ``provability''
operator $\Box=\wn\oc$ descends to an interior operator (in the sense of order theory) on the poset of
equivalence classes of c-formulas, whereas the ``solubility'' operator
$\nabla=\oc\wn$ descends to a closure operator (in the same sense) on the poset of equivalence
classes of i-formulas.
In the case of $\Box$, this amounts to (i) the derivability in QHC of the principles
\formulas
\begin{itemize}
\item[($1^\Box$)] $\prin\Box p\to p$;
\item[($2^\Box$)] $\prin\Box p\to\Box\Box p$;
\end{itemize}
and (ii) the judgement
\begin{itemize}
\item[($*$)] $\mm{\turnstile F\to G}$ implies $\mm{\turnstile\Box F\to\Box G}$.
\end{itemize}
These are easy to verify directly: ($1^\Box$) is the same as ($\wn\oc$);
($2^\Box$) follows from ($\oc\wn$) and the monotonicity of $\wn$;
and ($*$) follows from the monotonicity of $\oc$ and $\wn$.

In fact, ($*$) is a consequence of the derivability in QHC of the following principle and rule:
\begin{itemize}
\item[($3^\Box$)] $p\,/\,\Box p$;
\item[($4^\Box$)] $\prin\Box(p\to q)\To(\Box p\to\Box q)$;
\end{itemize}
Here ($3^\Box$) follows from ($\oc_\top$) and ($\wn_\top$),
and ($4^\Box$) from ($\octo$) and ($\wnto$).
We have proved

\begin{proposition} Sending $\Box$ to $\wn\oc$ yields a syntactic interpretation of QS4 in QHC, which 
is the identity on QC.
\end{proposition}

We will see in \S\ref{Euler-Tarski1} that this interpretation is exact. Before we get there, we need 
to distinguish two roles of the symbol ``$\Box$'': the modality of QS4 and an abbreviation for $\wn\oc$ in QHC. 

Similarly, that $\nabla$ induces a closure operator on the poset of
equivalence classes of i-formulas translates to (i) the derivability in QHC of the principles
\begin{itemize}
\item[($1^\nabla$)] $\prin\alpha\to\nabla\alpha$;
\item[($2^\nabla$)] $\prin\nabla\nabla\alpha\to\nabla \alpha$,
\end{itemize}
and (ii) the judgement
\begin{itemize} \metameta
\item[($*$)] $\turnstile\Phi\to\Psi$ implies $\turnstile\nabla\Phi\to\nabla\Psi$.
\end{itemize}\formulas
Here ($*$) is a consequence of the derivability in QHC of the principle
\begin{itemize}
\item[($4^\nabla$)]
$\prin\nabla(\alpha\to \beta)\To(\nabla \alpha\to\nabla \beta)$,
\end{itemize}
which follows from ($\octo$) and ($\wnto$).
We also note that the following consequence of ($\wn_\bot$) and ($\oc_\bot$),
\begin{itemize}
\item[($3^\nabla$)] $\nabla\ab\to\ab$,
\end{itemize}
is equivalent (modulo ($1^\nabla$) and ($4^\nabla$)) to $\neg\alpha\,/\,\neg\nabla\alpha$ (cf.\ \S\ref{negation} 
below), which can be considered to be dual to ($3^\Box$).
We define QH4 to be the logic obtained from QH by adding a new unary connective $\nabla$ and 
additional laws ($1^\nabla$)--($4^\nabla$).
We have thus proved:

\begin{proposition} \label{QH4}
Sending $\nabla$ to $\oc\wn$ yields a syntactic interpretation of QH4 in QHC, which is the identity on QH.
\end{proposition}
\metameta

It should be noted that the laws of QH4 mimic some properties of $\neg\neg$.
In fact, by substituting $\neg\neg$ for $\nabla$ we get
an interpretation of QH4 in QH.
Indeed, under this substitution, ($1^\nabla$) holds
by (\cite{M0}*{\S\ref{int:tautologies}, (\ref{int:double negation})}), ($2^\nabla$) and ($3^\nabla$) follow
from (\cite{M0}*{\S\ref{int:tautologies}, (\ref{int:triple negation})}), and ($4^\nabla$) holds
by (\cite{M0}*{\S\ref{int:tautologies}, (\ref{int:neg-neg-imp})}).
Let us note that since the purely intuitionistic fragment of QH4 is fixed under this interpretation,
this fragment is precisely QH (in other words, QH4 is a conservative extension of QH).
We will see in \S\ref{Tarski-Kolmogorov1} that the constructed interpretation of QH4 in QH
factors through the interpretation of \ref{QH4}.

It is not clear to the author whether the interpretation of \ref{QH4} is faithful (in other words, whether
QHC is a conservative extension of QH4).%
\footnote{Recently A. Onoprienko \cite{On1}, \cite{On2} affirmatively answered this question 
(which appeared already in the first arXiv version of the present paper).}
Thus one should not conflate two potentially distinct roles of the symbol ``$\nabla$'': the modality of
QH4 and an abbreviation for $\oc\wn$ in QHC.

The modal logic QH4 was studied by Fairtlough--Walton \cite{FW}, who called it QLL$^+$, and Aczel \cite{Ac},
who called it the logic of a strict lax modality (see also \cite{FM}*{Theorem 4.5}).
Later the zero-order fragment H4 was also studied by Art\"emov and Protopopescu, who showed its completeness with 
respect to some Kripke models \cite{AP} (beware that H4, which is called IEL$^+$ in \cite{AP}, disappeared 
from the published version of the preprint \cite{AP}).
The intuitionistic modal logic given by the laws ($1^\nabla$), ($2^\nabla$) and ($4^\nabla$)
was studied as early as 1950 by H. Curry \cite{Cu1}*{p.\ 120} (see also \cite{Cu2}*{\S5}), and later by 
Goldblatt \cite{Gold}*{\S14.5} and many others.
In particular, categorical models of QH4 related to the sheaf-valued models of QH in \cite{M0} are known; 
see \cite{FM}, \cite{Gold2}*{\S7.6}, \cite{AMPR}.

The properties of $\nabla$ are also similar to those of the squash/bracket operator in dependent type theory
(see \cite{AB} and references there).

\subsection{Simplification}

\formulas
\begin{proposition} \label{symmetry}
(a) The laws {\rm ($\wn_\land$)}, {\rm ($\wn_\lor$)}, {\rm ($\wn_\bot$)},
{\rm ($\wn_\forall$)} and {\rm ($\wn_\exists$)} are redundant.

(b) The following holds in QHC:
\begin{enumerate}
\item[\rm ($\oc_\land$)] $\turnstile\oc p\land\oc q\Tofrom \oc(p\land q)$;
\item[\rm ($\oc_\lor$)] $\turnstile\oc p\lor \oc q\To \oc(p\lor q)$;
\item[\rm ($\oc_\forall$)] $\turnstile\forall\tr x\,\oc p(\tr x)\Tofrom \oc\forall \tr x\,p(\tr x)$;
\item[\rm ($\oc_\exists$)] $\turnstile\exists\tr x\,\oc p(\tr x)\To \oc\exists \tr x\,p(\tr x)$.
\end{enumerate}
\end{proposition}
\metameta

\begin{remark}
From the informal semantic viewpoint, the implication $\fm{\turnstile\oc p\lor \oc q\too \oc(p\lor q)}$ cannot 
be reversed.
Indeed, let $P$ be the proposition {\it $i^i$ is a rational number}
and $Q$ the proposition {\it $i^i$ is an irrational real number}.
The problem $\oc(P\lor Q)$ amounts to showing that $i^i$ is a real number.
This problem is trivial: $i^i=(e^{i\pi/2})^i=e^{-\pi/2}$.
On the other hand, the problem $\oc P\lor\oc Q$ amounts to
$\oc(P\lor Q)\land\Gamma$, where $\Gamma$ is the problem
{\it Determine whether $e^{-\pi/2}$ is rational or irrational}.
This is not an easy problem.%
\footnote{In fact, $i^i$ is transcendental by the Gelfond--Schneider theorem.}
This shows incidentally that one cannot get an exact interpretation of
classical logic in intuitionistic logic by just looking at problems of
the form $\oc P$, where $P$ is a proposition.
\end{remark}

\formulas
\begin{proof}[Proof. Redundancy of {\rm ($\wn_\forall$)}]
By an intuitionistic law, $\turnstile\forall \tr x\,\alpha(\tr x)\to\alpha(\tr t)$.
Then by ($\wn_\top$) and ($\wnto$), we get $\turnstile\wn\forall \tr x\,\alpha(\tr x)\to\wn\alpha(\tr t)$.
By the classical generalization rule, we obtain 
$\turnstile\forall\tr t\,[\wn\forall \tr x\,\alpha(\tr x)\to\wn\alpha(\tr t)]$.
By another intuitionistic rule, we infer that
$\turnstile\wn\forall \tr x\,\alpha(\tr x)\to\forall\tr t\,\wn\alpha(\tr t)$.
Now the variable can be renamed.
\end{proof}

\begin{proof}[Redundancy of {\rm ($\wn_\land$)}]
The $\to$ implication in ($\wn_\land$) is redundant similarly to the redundancy of ($\wn_\forall$).
Conversely, the intuitionistic validity $\alpha\land\beta\to\alpha\land\beta$
can be rewritten, by the exponential law, as $\alpha\to(\beta\to(\alpha\land\beta))$.
Then by ($\wn_\top$) and ($\wnto$) it follows that
$\turnstile\wn\alpha\to(\wn\beta\to\wn(\alpha\land\beta))$.
Again applying the exponential law, this time regarded as an inference rule of classical logic,
we obtain $\turnstile\wn\alpha\land\wn\beta\to\wn(\alpha\land\beta)$.
\end{proof}

\begin{proof}[Proof of {\rm ($\oc_\exists$)} and {\rm ($\oc_\lor$)}]
This is parallel to the redundancy of ($\wn_\forall$).
In more detail, by a classical principle, $\turnstile p(\tr t)\to\exists \tr x\,p(\tr x)$.
Then by ($\oc_\top$) and ($\octo$), we get $\turnstile\oc p(\tr t)\to\oc\exists \tr x\,p(\tr x)$.
By the generalization rule, we obtain $\turnstile\forall\tr t\,[\oc p(\tr t)\to\oc\exists\tr x\,p(\tr x)]$.
By another classical rule, we get $\turnstile\exists\tr t\,\oc p(\tr t)\to\oc\exists\tr x\,p(\tr x)$.
Now the variable can be renamed.
The case of ($\oc_\lor$) is similar.
\end{proof}

\begin{proof}[Redundancy of {\rm ($\wn_\exists$)} and {\rm ($\wn_\lor$)}]
The $\from$ implication in ($\wn_\exists$) is redundant similarly to the proof of ($\oc_\exists$)
or to the redundancy of ($\wn_\forall$).
Conversely, by the proof of ($\oc_\exists$) we have shown that
$\turnstile\exists\tr x\,\oc p(\tr x)\to \oc\exists\tr x\, p(\tr x)$ using only ($\oc_\top$), ($\octo$) and classical logic.
Substituting, we get $\turnstile\exists\tr x\,\oc\wn\alpha(\tr x)\to\oc\exists\tr x\,\wn\alpha(\tr x)$.
On the other hand, from ($\oc\wn$) it follows that 
$\turnstile\exists\tr x\,\alpha(\tr x)\to\exists\tr x\,\oc\wn\alpha(\tr x)$.
By combining the two implications we get $\turnstile\exists\tr x\,\alpha(\tr x)\to\oc\exists\tr x\,\wn\alpha(\tr x)$.
By the proof of \ref{Galois}, we obtain from this the $\to$ implication in ($\wn_\exists$), using only
($\wn_\top$), ($\wnto$) and ($\wn\oc$).
The case of ($\wn_\lor$) is similar.
\end{proof}

\begin{proof}[Proof of {\rm ($\oc_\forall$)} and {\rm ($\oc_\land$)}]
The $\from$ implication in ($\oc_\forall$) is proved similarly to the redundancy of ($\wn_\forall$).
The converse implication is parallel to the redundancy of ($\wn_\exists$).
In more detail, ($\wn_\forall$) implies $\turnstile\wn\forall\tr x\,\oc p(\tr x)\to\forall\tr x\,\wn\oc p(\tr x)$, and
it follows from ($\wn\oc$) that $\turnstile\forall\tr x\,\wn\oc p(\tr x)\to\forall\tr x\,p(\tr x)$.
Thus $\turnstile\wn\forall\tr x\,\oc p(\tr x)\to \forall\tr x\,p(\tr x)$, hence by \ref{Galois}
$\turnstile\forall\tr x\,\oc p(\tr x)\to \oc\forall\tr x\,p(\tr x)$.
The case of ($\oc_\land$) is similar, or alternatively can be treated similarly to
the redundancy of ($\wn_\land$).
\end{proof}

\begin{proof}[Redundancy of {\rm ($\wn_\bot$)}]
By the explosion principle, we have $\turnstile\ab\to\oc\clbot$.
Then by ($\wn_\top$) and ($\wnto$) we get $\turnstile\wn\ab\to\wn\oc \clbot$.
On the other hand, by ($\wn\oc$) we have $\turnstile\wn\oc \clbot\to \clbot$.
Composing the two implications, we obtain $\turnstile\wn\ab\to \clbot$.
\end{proof}

\begin{corollary} \label{simplified}
The meta-conjunction of the following meta-formulas is a deductive system for QHC.

\begin{itemize}
\item A deductive system for intuitionistic logic;
\item A deductive system for classical logic;
\smallskip
\item[{\rm ($\oc_\top$)}] $\dfrac{p}{\oc p};$
\smallskip
\item[{\rm ($\wn_\top$)}] $\dfrac{\alpha}{\wn\alpha}$;
\medskip
\item[{\rm ($\wn\oc$)}] $\prin\wn\oc p\to p$;
\item[{\rm ($\oc\wn$)}] $\prin\alpha\to\oc\wn\alpha$;
\item[{\rm ($\octo$)}] $\prin\oc(p\to q)\To(\oc p\to \oc q)$;
\item[{\rm ($\wnto$)}] $\prin\wn(\alpha\to \beta)\To (\wn\alpha\to\wn\beta)$;
\item[{\rm ($\oc_\bot$)}] $\neg\oc \clbot$.
\end{itemize}
\end{corollary}
\metameta

\subsection{Negation} \label{negation}

\formulas
\begin{proposition} \label{v*}
Some laws of QHC can be rewritten as follows.

(a) {\rm ($\wn_\bot$)} is equivalent, modulo {\rm ($\wnto$)} and {\rm ($\wn_\top$)},
to $\prin\wn\neg\alpha\to\neg\wn\alpha$ and to $\neg\alpha\,/\,\neg\wn\alpha$;

(b) {\rm ($\oc_\bot$)} is equivalent, modulo {\rm ($\octo$)} and {\rm ($\oc_\top$)},
to $\prin\oc\neg p\to \neg \oc p$ and to $\neg p\,/\,\neg\oc p$.
\end{proposition}

\begin{proof}[Proof. (a)]
By ($\wnto$), we have $\turnstile\wn(\alpha\to\ab)\to (\wn\alpha\to\wn\ab)$.
Assuming ($\wn_\bot$), we also have $\turnstile\wn\ab\to \clbot$.
Hence $\turnstile\wn(\alpha\to\ab)\to (\wn\alpha\to \clbot)$; that is, $\turnstile\wn\neg\alpha\to\neg\wn\alpha$.

By ($\wn_\top$) we have $\neg\alpha\turnstile\wn\neg\alpha$.
Assuming $\prin\wn\neg\alpha\to\neg\wn\alpha$, by {\it modus ponens} we have
$\wn\neg\alpha\turnstile\neg\wn\alpha$.
Combining these yields $\neg\alpha\turnstile\neg\wn\alpha$.

Finally, assuming $\neg\alpha\,/\,\neg\wn\alpha$, we have, in particular,
$\neg\ab\turnstile\neg\wn\ab$.
Since $\neg\ab=\ab\to\ab$ is an intuitionistic validity, we get $\turnstile\neg\wn\ab$.
\end{proof}

\begin{proof}[(b)] By ($\octo$), we have $\turnstile\oc(p\to \clbot)\to (\oc p\to\oc \clbot)$.
Assuming ($\oc_\bot$), we also have $\turnstile\oc \clbot\to\ab$.
Hence $\turnstile\oc(p\to \clbot)\to (\oc p\to\ab)$; that is, $\turnstile\oc\neg p\to\neg\oc p$.

By ($\oc_\top$) we have $\neg p\turnstile\oc\neg p$.
Assuming $\prin\oc\neg p\to\neg\oc p$, by {\it modus ponens} we have
$\oc\neg p\turnstile\neg\oc p$.
Combining these yields $\neg p\turnstile\neg\oc p$.

Finally, assuming $\neg p\,/\,\neg\oc p$, we have, in particular,
$\neg \clbot\turnstile\neg\oc \clbot$.
Since $\neg \clbot=\clbot\to \clbot$ is a classical validity, we get $\turnstile\neg\oc \clbot$.
\end{proof}

\begin{proposition} \label{insolubility} $\turnstile\neg\alpha\tofrom\oc\neg\wn\alpha$.
\end{proposition}

This yields a definition of intuitionistic negation in terms of classical one.
As discussed in detail in \cite{M0}*{\S\ref{int:about-bhk}}, this fully agrees with
the BHK interpretation (and with a remark by Heyting; but disagrees with a remark by Kolmogorov).
Thus, this is yet another feature of the BHK interpretation that is captured in QHC but not
in the usual formalization of intuitionistic logic.

\begin{proof} By ($\oc\wn$), $\turnstile\neg\alpha\to\oc\wn\neg\alpha$, from \ref{v*}
we get $\turnstile\oc\wn\neg\alpha\to \oc\neg\wn\alpha$ and
$\turnstile\oc\neg\wn\alpha\to \neg\oc\wn\alpha$, and by
the contrapositive of ($\oc\wn$), $\turnstile\neg\oc\wn\alpha\to\neg\alpha$.
\end{proof}

\begin{remark}\label{insolubility-2}
Since $\turnstile\neg\neg(\alpha\lor\neg\alpha)$ (see \cite{M0}*{(\ref{int:not-not-LEM})}),
by \ref{insolubility} and by the converse of ($\oc_\top$), we have
$\turnstile\neg\Box\neg\wn(\alpha\lor\neg\alpha)$; thus it is impossible to prove
that $\alpha\lor\neg\alpha$ has no solutions.
\end{remark}

\begin{corollary}\label{weak Hilbert}
$\turnstile\wn\neg\alpha$ if and only if $\turnstile\neg\wn\alpha$.
\end{corollary}

Here the ``only if'' part is a consequence of \ref{v*}(a).
The ``if'' part can also be stated in a stronger form: $\neg\wn\alpha\turnstile\wn\neg\alpha$.

\begin{proof} Indeed, we have $\neg\wn\alpha\turnstile\wn\oc\neg\wn\alpha$ by ($\oc_\top$) and ($\wn_\top$),
and $\turnstile\wn\oc\neg\wn\alpha\tofrom\wn\neg\alpha$ by \ref{insolubility}.
\end{proof}

\begin{corollary} \label{move-nabla}
$\turnstile\neg\nabla\alpha\Tofrom\neg \alpha$ and
$\turnstile\neg\alpha\Tofrom\nabla\neg \alpha$.
\end{corollary}

This follows from the proof of \ref{insolubility}.

We note that Corollary \ref{move-nabla} implies that $\turnstile\neg\neg \alpha\Tofrom\neg\nabla\neg \alpha$
and $\turnstile\neg\nabla\neg \alpha\Tofrom\neg(\neg\nabla\neg)\neg \alpha$, which is in contrast with 
$\turnstile\neg(\neg\Box\neg)\neg p\tofrom\Box p$.

\begin{corollary} \label{nabla-negneg1}
$\turnstile\nabla \alpha\to\neg\neg \alpha$.
\end{corollary}

\begin{proof}
By \ref{move-nabla}, $\turnstile\neg \alpha\to\neg\nabla \alpha$.
Then \cite{M0}*{\S\ref{int:tautologies}, (\ref{int:implication-shift})} yields $\turnstile\nabla \alpha\to\neg\neg \alpha$.
\end{proof}

In fact, $\prin\nabla\alpha\to\neg\neg\alpha$ is yet another equivalent form of the law
($\oc_\bot$), since $\nabla\ab\to\neg\neg\ab$ implies $\nabla\ab\to\ab$, or $\oc \clbot\to\ab$.
Moreover, as observed in \cite{Ac}, $\prin\nabla\alpha\to\neg\neg\alpha$ is also an equivalent form
of the law (3$^\nabla$) of QH4.

\begin{remark} Using \ref{insolubility}, the following consequence can be drawn from the fact that
the implication of ($\octo$) goes, in a sense, in the opposite direction with respect to that of
($\oc_\lor$) and with respect to one of the implications of ($\oc_\land$).
The intuitionistic implications
$\turnstile\alpha\lor\beta\to\neg(\neg\alpha\land\neg\beta)$ and $\turnstile\alpha\lor\beta\to\neg\alpha\to\beta$
(cf.\ \cite{M0}*{(\ref{int:double negation}), (\ref{int:deMorgan2}) and (\ref{int:implication0'})}),
when specialized to the image of $\oc$, i.e., in the form $\turnstile\oc p\lor\oc q\to\neg\oc p\to\oc q$ and
$\turnstile\oc p\lor\oc q\to\neg(\neg\oc p\land\neg\oc q)$,
each factor into two irreversible (as we will see in \cite{M2}) implications in QHC:
$\turnstile\oc\Box p\lor\oc q\To\oc(\Box p\lor q)$ and $\turnstile\oc(\neg\Box p\to q)\To \oc\neg\Box p\to\oc q$;
respectively,
$\turnstile\oc\Box p\lor\oc\Box q\To\oc(\Box p\lor\Box q)$ and
$\turnstile\oc\neg(\neg\Box p\land\neg\Box q)\To\neg(\oc\neg\Box p\land\oc\neg\Box q)$.
\end{remark}

\metameta

\subsection{Implication}

\formulas
\begin{proposition}\label{move-oc-wn}
(a) $\turnstile(\oc\wn\alpha\to\oc\wn\beta)\Tofrom\oc(\wn\alpha\to\wn\beta)$;

(b) $\turnstile\wn\oc p\to\wn\oc q\Iff\wn(\oc p\to\oc q)$.
\end{proposition}

\noindent
We will actually prove stronger assertions: 
\begin{enumerate}[label=(\alph*)]
\item $\turnstile(\oc\wn\alpha\to\oc q)\To\oc(\wn\alpha\to q)$;
\item $\wn\alpha\to\wn\oc q\turnstile\wn(\alpha\to\oc q)$.
\end{enumerate}
(Their converses follow from ($\octo$) and ($\wnto$), respectively.)

\begin{proof}[Proof. (a)]
By ($\wnto$), $\turnstile\wn(\oc\wn\alpha\to\oc q)\to (\wn\oc\wn\alpha\to \wn\oc q)$.
Since $\turnstile \wn\oc\wn\alpha\tofrom\wn\alpha$ and
$\turnstile \wn\oc q\to q$, we get $\turnstile\wn(\oc\wn\alpha\to\oc q)\to(\wn\alpha\to q)$.
Then by \ref{Galois},
$\turnstile (\oc\wn\alpha\to \oc q)\To \oc(\wn\alpha\to q)$.
\end{proof}

\begin{proof}[(b)] By ($\oc_\top$), $\wn\alpha\to\wn\oc q\turnstile\oc(\wn\alpha\to\wn\oc q)$,
and by (a), $\turnstile\oc(\wn\alpha\to\wn\oc q)\to(\oc\wn\alpha\to\oc\wn\oc q)$.
Since $\turnstile \alpha\to\oc\wn\alpha$ and $\turnstile\oc\wn\oc q\tofrom q$, we get
$\turnstile\oc(\wn\alpha\to\wn\oc q)\to(\alpha\to\oc q)$.
Finally, by ($\wn_\top$), $\alpha\to\oc q\turnstile\wn(\alpha\to\oc q)$.
\end{proof}

The following proposition strengthens \ref{sup-inf}(b).
In addition, its part (a) along with part (a) of the preceding proposition generalize
\ref{insolubility} and \ref{move-nabla}.

\begin{proposition}\label{move-oc-wn2} We have

(a) $\turnstile(\nabla\alpha\to\nabla\beta)\Tofrom(\alpha\to\nabla\beta)$
and $\turnstile\nabla(\alpha\to\nabla\beta)\Tofrom(\alpha\to\nabla\beta)$;

(b) $\turnstile\Box p\to\Box q\Iff\Box p\to q$ and $\turnstile\Box(\Box p\to q)\Iff\Box p\to q$.
\end{proposition}

\begin{proof}[(a)] Since $\turnstile\alpha\to\nabla\alpha$, we get
$\turnstile(\nabla\alpha\to\nabla\beta)\To(\alpha\to\nabla\beta)$
and $\turnstile(\alpha\to\nabla\beta)\To\nabla(\alpha\to\nabla\beta)$.
Finally, $\turnstile\nabla(\alpha\to\nabla\beta)\To(\nabla\alpha\to\nabla\beta)$ 
by ($4^\nabla$) and ($2^\nabla$).
\end{proof}

\begin{proof}[(b)] By ($1^\Box$), $\turnstile(\Box p\to\Box q)\To(\Box p\to q)$. 
By ($3^\Box$), $\Box p\to q\turnstile\Box(\Box p\to q)$.
Finally, by ($4^\Box$) and ($2^\Box$), $\turnstile\Box(\Box p\to q)\To(\Box p\to\Box q)$.
\end{proof}

\metameta

\begin{corollary}\label{Russell-Prawitz} If $\Phi$ is an i-formula, $[\nabla\Phi]$ is the least upper bound of 
all classes $[(\Phi\to\oc F)\to\oc F]$, where $F$ is a c-formula.
\end{corollary}

Let us note that $(\Phi\to\oc F)\to\oc F$ specializes to $\neg\neg\Phi$ when $F=\clbot$.

\begin{proof} Let us observe that $[\Psi]$ is an upper bound of all $[(\Phi\to\oc F)\to\oc F]$ if and only if
$\turnstile\Psi\to\big((\Phi\to\oc F)\to\oc F\big)$ for all c-formulas $F$.
By the exponential law, the latter is equivalent to $\turnstile (\Phi\to\oc F)\to(\Psi\to\oc F)$.
Now by \ref{move-oc-wn2}(a) we do have $\turnstile (\Phi\to\oc F)\to(\nabla\Phi\to\oc F)$ for all
c-formulas $F$.
It remains to show that if $\turnstile (\Phi\to\oc F)\to(\Psi\to\oc F)$ for all c-formulas $F$,
then $\turnstile\Psi\to\nabla\Phi$.
Indeed, this follows by setting $F=\wn\Phi$.
\end{proof}

A variation of \ref{Russell-Prawitz} can be formulated within the meta-logic, similarly to 
\cite{M0}*{\ref{int:RP-translation2}}:

\formulas
\begin{corollary} \label{Russell-Prawitz2} $\turnstile\nabla\alpha\Iff\mq{p}(\alpha\to\oc p)\to\oc p$.
\end{corollary}

Of course, if we replace $\iff$ by $\tofrom$ here, we will get a meaningless expression (i.e., 
not a well-typed $\mm\lambda$-expression) since the right hand side contains a meta-quantifier.
But if we could do this, then \ref{Russell-Prawitz2} would be saying that $\nabla$ is 
a ``Russell--Prawitz modality'' in the terminology of Aczel \cite{Ac} (see also \cite{Coq}).
 
\begin{proof} Since $p$ does not occur in $\nabla\alpha$, to show that 
$\nabla\alpha\turnstile\mq{p}(\alpha\to\oc p)\to\oc p$,
it suffices to show that $\nabla\alpha\turnstile(\alpha\to\oc p)\to\oc p$ 
(by the generalization meta-rule).
This in turn reduces to deriving $\nabla\alpha\to\big((\alpha\to\oc p)\to\oc p\big)$.
By the exponential law the latter formula is equivalent to $(\alpha\to\oc p)\to(\nabla\alpha\to\oc p)$,
which was derived in \ref{move-oc-wn2}(a).

Conversely, by the specialization meta-rule,
$\mq{p}(\alpha\to\oc p)\to\oc p\turnstile(\alpha\to\oc\wn\alpha)\to\oc\wn\alpha$.
But the latter formula is equivalent to $\nabla\alpha$ due to $\oc\wn$.
\end{proof}

\metameta

\subsection{Distributivity properties}
\formulas

The following is a direct consequence of \ref{symmetry}(b).

\begin{proposition}\label{move-box-diamond}
$\turnstile\Box(p\land q)\Tofrom\Box p\land\Box q$
\ and \
$\turnstile\nabla(\alpha\land \beta)\Tofrom\nabla \alpha\land\nabla \beta$.
\end{proposition}

\begin{proposition} \label{McKinsey+Gentzen}
The following holds in QHC.

(a) $\turnstile\Box(\wn\alpha\land \wn\beta)\Tofrom \wn\alpha\land \wn\beta$;

(b) $\turnstile\Box(\wn\alpha\lor \wn\beta)\Tofrom \wn\alpha\lor \wn\beta$;

(c) $\turnstile\Box\exists\tr x\, \wn\alpha(\tr x)\Tofrom\exists\tr x\, \wn\alpha(\tr x)$;

(d) $\turnstile\nabla(\oc p\land \oc q)\Tofrom \oc p\land \oc q$;

(e) $\turnstile\nabla(\oc p\to \oc q)\Tofrom \oc p\to \oc q$;

(f) $\turnstile\nabla\forall\tr x\, \oc p(\tr x)\Tofrom\forall\tr x\, \oc p(\tr x)$.
\end{proposition}
		
It is easy to see that these assertions are equivalent to their special
cases for i-formulas in the image of $\oc$ and for c-formulas in
the image of $\wn$.
Those special cases are in turn parallel to \cite{M0}*{\ref{int:QS4-McKinsey} and
\ref{int:Goedel-Gentzen}}.

\begin{proof}[Proof. (a,d)]
These follow from \ref{move-box-diamond} using
$\turnstile \wn\oc\wn\alpha\tofrom \wn\alpha$ or $\turnstile \oc\wn\oc p\tofrom \oc p$.
\end{proof}

\begin{proof}[(a,d,b,c,f)] Let us check (b).
We have $\turnstile\wn\alpha\lor\wn\beta\tofrom\wn(\alpha\lor \beta)$ and
$\turnstile\wn\oc\wn(\alpha\lor \beta)\tofrom \wn\oc(\wn\alpha\lor \wn\beta)$.
\end{proof}

\begin{proof}[(e)] ``$\to$'' follows from ($4^\nabla$), ($\oc\wn$) and ($\wn\oc$).
The converse follows from ($\oc\wn$).
\end{proof}

\begin{proposition} \label{Goedel+Kuroda}
The following holds in QHC.

(a) $\turnstile\wn(\alpha\land \beta)\Tofrom \wn(\nabla \alpha\land\nabla \beta)$;

(b) $\turnstile\wn(\alpha\lor \beta)\Tofrom \wn(\nabla \alpha\lor\nabla \beta)$;

(c) $\turnstile\wn\exists\tr x\, \alpha(\tr x)\Tofrom \wn\exists\tr x\,\nabla \alpha(\tr x)$;

(d) $\turnstile\oc(p\land q)\Tofrom \oc(\Box p\land\Box q)$;

(e) $\turnstile\oc\forall\tr x\, p(\tr x)\Tofrom \oc\forall\tr x\,\Box p(\tr x)$.
\end{proposition}

Applying $\oc$ to both sides in (a), (b), (c), and $\wn$ to both sides
in (d), (e) leads to no loss of generality, but makes the
validities parallel to \cite{M0}*{\ref{int:QS4-Goedel}
and \ref{int:Kuroda}} --- with the exception of one ``missing validity'',
$\wn(\alpha\to\beta)\Tofrom\wn(\nabla\alpha\to\nabla\beta)$, which will turn out to be
an independent principle \cite{M2}*{\ref{g2:principles}(a) and \ref{g2:sheaf-refutation}}.

\begin{proof} Assertions (a,b,c) follow from ($\wn_\land$), ($\wn_\lor$)
and ($\wn_\exists$) using that $\turnstile \wn\alpha\tofrom \wn\oc\wn\alpha$.
Assertions (d,e) follow from ($\oc_\land$) and ($\oc_\forall$) using that $\turnstile \oc p\tofrom \oc\wn\oc p$.
\end{proof}

\begin{proposition}\label{worst}
The following holds in QHC.

(a) $\turnstile\wn\oc(\wn\alpha\lor \wn\beta)\Tofrom \wn(\oc\wn\alpha\lor\oc\wn\beta)$;

(b) $\turnstile\wn\oc\exists\tr x\, \wn\alpha(\tr x)\Tofrom \wn\exists\tr x\, \oc\wn\alpha(\tr x)$;

(c) $\turnstile\oc(\wn\oc p\to \wn\oc q)\Tofrom \oc\wn(\oc p\to \oc q)$;

(d) $\turnstile\oc\forall\tr x\, \wn\oc p(\tr x)\Tofrom \oc\wn\forall\tr x\, \oc p(\tr x)$.
\end{proposition}

\begin{proof}[Proof. (a,b)]
On applying ($\wn_\lor$) or ($\wn_\exists$) to the right hand side, these
reduce to \ref{McKinsey+Gentzen}(b,c).
\end{proof}

\begin{proof}[(c,d)]
On applying \ref{move-oc-wn}(a) or ($\oc_\forall$)
to the left hand side, these reduce to \ref{McKinsey+Gentzen}(e,f).
\end{proof}

\metameta

\section{Stability and decidability}\label{stable and decidable}

\subsection{Stable and decidable c-formulas}
Let us recall that a i-formula $\Phi$ is called decidable if
$\turnstile \Phi\lor\neg\Phi$, and stable if
$\turnstile\neg\neg\Phi\to\Phi$; decidable i-formulas are stable (see \cite{M0}*{(\ref{int:decidable-stable})}).
Let us call a c-formula $F$ {\it decidable} if
$\turnstile \oc F\lor \oc\neg F$, and {\it stable} if
$\turnstile\neg\oc\neg F\to\oc F$; decidable c-formulas are stable
(using the intuitionistic law $\fm{\prin\alpha\lor\beta\to\neg\beta\to\alpha}$, cf.\
\cite{M0}*{(\ref{int:implication0'})}).
Let us note that by \ref{Galois}, $\turnstile\neg\oc\neg F\to\oc F$ is equivalent to
$\turnstile\wn\neg\oc\neg F\to F$, which by \ref{insolubility} is in turn equivalent to
$\turnstile\wn\oc\neg\wn\oc\neg F\to F$, that is, $\turnstile\Box\Diamond F\to F$.
In words, ``if $F$ is provably irrefutable, then it is true''.

Usually the notions of stability and decidability are considered relatively to a theory over
intuitionistic logic.
Instead of doing this we will consider internalizations of stability and decidability as operators.

\formulas
Thus we define $\Dec,\Stab:\1_i\too\1_i$ as $\alpha\mapsto\alpha\lor\neg\alpha$ and 
$\alpha\mapsto\neg\neg\alpha\to\alpha$ respectively;
and $\Dec,\Stab:\1_c\too\1_i$ as $p\mapsto\oc p\lor\oc\neg p$ and 
$p\mapsto\neg\oc\neg p\to\oc p$ respectively.
\metameta

\begin{proposition}\label{stable-decidable}
(a) If a c-formula $F$ is stable or decidable, then so is the i-formula $\oc F$.
The converse holds for c-formulas $F$ of the form $\wn\Phi$.

(b) If an i-formula $\Phi$ is stable or decidable, then so is
the c-formula $\wn\Phi$.
The converse holds for i-formulas $\Phi$ of the form $\oc F$.
\end{proposition}

These judgements about the QHC calculus follow from their internalized versions, which will be proved below:
 
\formulas
(a) $\turnstile\Dec(p)\to\Dec(\oc p)$ and $\turnstile\Stab(p)\to\Stab(\oc p)$.

\noindent
Moreover, $\turnstile\Dec(\wn\alpha)\tofrom\Dec(\oc\wn\alpha)$ and 
$\turnstile\Stab(\wn\alpha)\tofrom\Stab(\oc\wn\alpha)$.

(b) $\turnstile\Dec(\alpha)\to\Dec(\wn\alpha)$ and $\turnstile\Stab(\alpha)\to\Stab(\wn\alpha)$.

\noindent
Moreover, $\turnstile\Dec(\oc p)\to\Dec(\wn\oc p)$ and $\turnstile\Stab(\oc p)\to\Stab(\wn\oc p)$.

\begin{proof}[Proof. (a)] The first assertion follows since $\turnstile\oc\neg p\to\neg\oc p$ by \ref{v*}(b).
The moreover assertion follows since $\turnstile\oc\neg\wn\alpha\tofrom\neg\oc\wn\alpha$ by 
\ref{insolubility} and \ref{move-nabla}.
\end{proof}

\begin{proof}[(b)] By \ref{move-nabla},
$\turnstile \neg\nabla\alpha\tofrom\nabla\neg\alpha$.
From this and ($1^\nabla$) or ($4^\nabla$) it follows that
$\turnstile\Dec(\alpha)\to\Dec(\nabla\alpha)$ and $\turnstile\Stab(\alpha)\to\Stab(\nabla\alpha)$.
Now the first assertion of (b) follows from the moreover assertion of (a).
The moreover assertion of (b) follows from the first assertion of (a).
\end{proof}
\metameta

\begin{proposition}\label{stable-decidable2}
(a) C-formulas of the form $\neg\wn\Phi$ are stable.

(b) If $\Phi$ is stable, then $\wn\Phi$ is decidable if and only if
$\Phi$ is decidable.

(c) If $\neg F$ is stable, then $\oc F$ is decidable if and only if $F$
is decidable.
\end{proposition}

\formulas
We will prove the internalizations: (a) $\turnstile\Stab(\neg\wn\alpha)$;

(b) $\turnstile\Stab(\alpha)\to\big(\Dec(\wn\alpha)\tofrom\Dec(\alpha)\big)$;

(c) $\turnstile\Stab(\neg p)\to\big(\Dec(\oc p)\tofrom\Dec(p)\big)$.

\begin{proof}[Proof. (a)]
By the classical double negation law,
$\turnstile\neg\oc\neg\neg\wn\alpha\tofrom\neg\oc\wn\alpha$ and by \ref{insolubility} and \ref{move-nabla},
also $\turnstile\neg\oc\wn\alpha\tofrom\oc\neg\wn \alpha$.
Thus $\turnstile\neg\oc\neg(\neg\wn\alpha)\to\oc(\neg\wn \alpha)$.
\end{proof}

\begin{proof}[(b)] Assuming $\alpha\tofrom\neg\neg\alpha$, and writing $\beta=\neg\alpha$,
from $\turnstile\nabla\neg\beta\lor\neg\nabla\neg\beta\tofrom\neg\beta\lor\neg\neg\beta$
we get $\nabla\alpha\lor\neg\nabla\alpha\tofrom\alpha\lor\neg\alpha$.
This shows that $\turnstile\Stab(\alpha)\to\big(\Dec(\nabla\alpha)\tofrom\Dec(\alpha)\big)$,
and the assertion now follows from \ref{stable-decidable}(b).
\end{proof}

\begin{proof}[(c)] Clearly, $\turnstile\Dec(\neg p)\tofrom(\neg\oc p\tofrom\oc\neg p)$,
and the assertion follows.
\end{proof}
\metameta

\begin{proposition} \label{nabla-negneg2}
$\turnstile\neg\neg\Phi\tofrom\nabla\Phi$ if and only if $\wn\Phi$ is stable.
\end{proposition}

\formulas
We will prove the internalization:
$\turnstile (\neg\neg\alpha\tofrom\nabla\alpha)\tofrom\Stab(\wn\alpha)$.

\begin{proof} $\nabla\alpha\to\neg\neg\alpha$ is derivable \ref{nabla-negneg1}.
The converse implication, $\neg\neg\alpha\to\nabla\alpha$, is equivalent by \ref{move-nabla}
to $\Stab(\nabla\alpha)$, which by \ref{stable-decidable}(b) is in turn equivalent to $\Stab(\wn\alpha)$.
\end{proof}
\metameta

\subsection{Semi-stability and semi-decidability}
Let us call an i-formula $\Phi$ {\it semi-decidable} if $\turnstile\wn\Dec(\Phi)$, and {\it semi-stable} if
$\turnstile\wn\Stab(\Phi)$.
Similarly, we call a c-formula $F$ {\it semi-decidable} if $\turnstile\wn\Dec(F)$, and {\it semi-stable} if
$\turnstile\wn\Stab(F)$.
Here each ``$\wn$'' can be replaced by ``$\nabla$'' due to $\fm{\turnstile\gamma\iff\oc\gamma}$.
Stability or decidability implies semi-stability or semi-decidability (both for i-formulas and for
c-formulas) due to $\fm{\gamma\turnstile\wn\gamma}$.
Semi-decidability implies semi-stability (both for i-formulas and c-formulas) for the same
reasons that decidability implies stability.

\begin{remark}\label{stable-decidable-bis}
Since \ref{stable-decidable} holds in the internalized form, we can
apply ($\wn_\top$) and ($\wnto$) to obtain the literal analogue of \ref{stable-decidable}
for semi-stability and semi-decidability, in the internalized form.
\end{remark}

\begin{proposition} \label{negation-commutes}
(a) An i-formula $\Phi$ is semi-decidable if and only if $\turnstile \wn\neg\Phi\tofrom\neg\wn\Phi$.

(b) A c-formula $\neg F$ is stable if and only if $\turnstile \oc\neg F\tofrom\neg\oc F$.
\end{proposition}

These hold internally:
\formulas

(a) $\turnstile\wn\Dec(\alpha)\tofrom(\wn\neg\alpha\tofrom\neg\wn\alpha)$;

(b) $\turnstile\Stab(\neg p)\tofrom(\oc\neg p\tofrom\neg\oc p)$.

\noindent
Part (b) is trivial.

\begin{proof}[Proof of (a)]
$\neg\wn\alpha\to\wn\neg\alpha$ is classically equivalent to $\turnstile\wn\alpha\lor\wn\neg\alpha$, which is 
in turn equivalent to $\turnstile\wn(\alpha\lor\neg\alpha)$.
\end{proof}

\metameta

\begin{proposition}\label{semi-stable-decidable}
(a) A c-formula $F$ is semi-stable if and only if it is stable.

(b) A i-formula $\Phi$ is semi-decidable if and only if the c-formula $\wn\Phi$ is.
\end{proposition}

These hold internally:

\formulas
(a) $\turnstile\wn\Stab(p)\iff\Stab(p)$;

(b) $\turnstile\wn\Dec(\wn\alpha)\tofrom\wn\Dec(\alpha)$.

\begin{proof}[Proof. (a)] 
By ($\oc_\top$), $\wn\Stab(p)\turnstile\nabla\Stab(p)$.
On the other hand, by ($4^\nabla$) and \ref{move-nabla} we also have
$\turnstile\nabla(\neg\oc\neg p\to\oc p)\to(\neg\oc\neg p\to\oc p)$, that is,
$\turnstile\nabla\Stab(p)\to\Stab(p)$.
\end{proof}

\begin{proof}[(b)] Using ($\wn_\lor$) and \ref{insolubility}, we get
$\turnstile\wn(\oc\wn\alpha\lor\oc\neg\wn\alpha)\tofrom(\wn\alpha\lor\wn\neg\alpha)$, and using ($\wn_\lor$) again,
we get $\turnstile(\wn\alpha\lor\wn\neg\alpha)\tofrom\wn(\alpha\lor\neg\alpha)$, as desired.
\end{proof}
\metameta

\formulas
\begin{corollary} \label{nabla-stable-decidable}
$\turnstile\nabla\Dec(p)\to\Stab(p)$.
\end{corollary}

This is a strengthening of ``decidability implies stability'' for c-formulas.

\begin{proof}
Decidability does imply stability: $\turnstile\Dec(p)\to\Stab(p)$.
Hence $\turnstile\nabla\Dec(p)\to\nabla\Stab(p)$.
On the other hand, by \ref{semi-stable-decidable}(a), $\turnstile\nabla\Stab(p)\to\Stab(p)$.
\end{proof}

\begin{corollary} \label{pushout formula}
$\turnstile\nabla(\nabla\alpha\lor\neg\nabla\alpha)\To(\nabla\alpha\tofrom\neg\neg\alpha)$.
\end{corollary}

This will be used in \S\ref{Tarski-Kolmogorov1} to show that the classical $\neg\neg$-translation of
QC into QH cannot be improved in a certain sense.

\begin{proof} By \ref{stable-decidable}(a), $\turnstile\nabla\Dec(\nabla\alpha)\tofrom\nabla\Dec(\wn\alpha)$.
By \ref{nabla-stable-decidable}, $\turnstile\nabla\Dec(\wn\alpha)\to\Stab(\wn\alpha)$.
By \ref{nabla-negneg2}, $\turnstile\Stab(\wn\alpha)\tofrom (\neg\neg\alpha\tofrom\nabla\alpha)$.
\end{proof}

\metameta

\section{Syntactic interpretations} \label{syntactic}

\subsection{$\Box$-interpretation}\label{Euler-Tarski1}
The classical provability translation of QH in QS4 (see \cite{M0}*{\S\ref{int:provability}})
extends to the following syntactic {\it $\Box$-interpretation} of QHC in QS4, denoted by
$A\mapsto A_\Box$:
\begin{itemize}
\item Atomic c-formulas and classical connectives remain unchanged;
\item Atomic i-formulas are re-typed as atomic c-formulas and are prefixed by $\Box$;
\item Intuitionistic $\land$, $\lor$ and $\exists$ become classical, and $\ab$, $\triv$ are replaced by 
$\clbot$, $\cltop$;
\item Intuitionistic $\to$ and $\forall$ become classical and are prefixed by $\Box$;
\item $\wn$ is erased, and $\oc$ is replaced by $\Box$.
\end{itemize}

\formulas
Indeed, let us write out the images of the laws and inference rules in \ref{simplified}
under the $\Box$-interpretation:

($\wn_\top$) $\alpha\,/\,\wn\alpha$ becomes $\Box a\,/\,\Box a$;

($\oc_\top$) $p\,/\,\oc p$ becomes $p\,/\,\Box p$;

($\wn\oc$) $\prin\wn\oc p\to p$ becomes $\prin\Box p\to p$;

($\oc\wn$) $\prin\alpha\to\nabla\alpha$ becomes $\prin\Box(\Box a\to\Box\Box a)$;

($\octo$) $\prin\oc(p\to q)\to(\oc p\to\oc q)$ becomes
$\prin\Box\big(\Box(p\to q)\to\Box(\Box p\to\Box q)\big)$;

($\wnto$) $\prin\wn(\alpha\to\beta)\to(\wn\alpha\to\wn\beta)$ becomes 
$\prin\Box(\Box a\to\Box b)\to(\Box a\to\Box b)$;

($\oc_\bot$) $\prin\oc\clbot\to\clbot$ becomes $\prin\Box(\Box \clbot\to \clbot)$.

\noindent
The resulting rules and principles are easily derivable in QS4, including the last one, which is the principle of 
internal consistency (see \cite{M0}*{\S\ref{int:provability}}).
\metameta

The classical laws and inference rules of QHC hold under the $\Box$-interpretation since
it does nothing to classical connectives and quantifiers and to atomic c-formulas.
The intuitionistic laws and inference rules of QHC hold under the $\Box$-interpretation
since the restriction of the $\Box$-interpretation to QH is known to be
an interpretation (see \cite{M0}*{\S\ref{int:provability}}).

Finally, let us note that by an inductive argument based on
\cite{M0}*{\ref{int:QS4-McKinsey}}, $\turnstile\Phi_\Box\tofrom\Box\Phi_\Box$ for any i-formula $\Phi$.
It follows that the second-order meta-specialization of the intuitionistic type holds under 
the $\Box$-interpretation.
The other meta-rules hold under the $\Box$-interpretation for trivial reasons.

We have proved

\begin{theorem} \label{box-int} (a) If $A_1,\dots,A_n\turnstile _{QHC} A$, then
$(A_1)_\Box,\dots,(A_n)_\Box\turnstile _{QS4} A_\Box$, and the converse holds (trivially)
if $A_1,\dots,A_n,A$ are formulas of QS4.

(b) If a formula $A$ is derivable in QHC from $\prin A_1,\dots,\prin A_n$, then $A_\Box$
is derivable in QS4 from $\prin(A_1)_\Box,\dots,\prin(A_n)_\Box$.
\end{theorem}

Of course, by \cite{M0}*{\ref{int:QS4-McKinsey}}, all of the intuitionistic
connectives and quantifiers (and not only $\to$ and $\forall$) could be prefixed by
a $\Box$ in the definition of the $\Box$-interpretation.
Consequently, by \cite{M0}*{\ref{int:QS4-Goedel}}, one could alternatively formulate
the $\Box$-translation as an extension of G\"odel's original
provability translation: {\it post}fix by $\Box$'es the intuitionistic
$\lor$, $\exists$ and $\to$, erase every $\oc$, and replace every $\wn$
by a $\Box$.
Like before, all intuitionistic connectives and quantifiers become classical, and
all atomic i-formulas are re-typed as atomic c-formulas.

\begin{theorem} \label{conservative}
The QHC calculus is:

(a) a strongly conservative extension of classical predicate calculus QC;

(b) a strongly conservative extension of QS4, via $\Box\mapsto\wn\oc$.
\end{theorem}

Here a formula of QHC is regarded as a formula of QS4 if it involves only classical atoms, connectives and
quantifiers, as well as the combination $\Box=\wn\oc$ (but not $\wn$ and $\oc$ alone).

Strong conservativity in (b) means that if a derivable rule of QHC is expressed in the language of QS4, then 
it is derivable in QS4.

\begin{proof} Since the standard interpretation of QS4 in QHC
(see \ref{modalities}) composed with the $\Box$-interpretation
is the identity, we get (b).
Omitting each $\Box$ is clearly an interpretation of QS4 in QC
that restricts to the identity on QC.
Thus QS4 is a strongly conservative extension of QC, and we obtain (a).
\end{proof}

\subsection{$\nabla$-interpretation}\label{nabla}

By Theorem \ref{conservative}(b), the $\Box$-interpretation of QHC in
QS4 can be regarded as an interpretation of QHC in itself.
This does not preserve the types of formulas (i.e., i-formulas
versus c-formulas), but can be amended to do so.
This results in the following {\it $\nabla$-interpretation} of QHC
in itself, which restricts to an unintended embedding of QH in QHC:

\begin{theorem} \label{nabla-interpretation}
If $A$ is a formula of QHC, let $A_\nabla$ be the formula
of QHC obtained from $A$ by prefixing atomic i-formulas and
the intuitionistic $\lor$ and $\exists$ by $\nabla:=\oc\wn$.
Then 

(a) $\turnstile A$ implies $\turnstile A_\nabla$, and the converse holds when $A$ is a formula of QH
or (trivially) of QS4;

(b) if $A_1,\dots,A_n\turnstile A$, then $(A_1)_\nabla,\dots,(A_n)_\nabla\turnstile A_\nabla$;

(c) if a formula $A$ is derivable in QHC from $\prin A_1,\dots,\prin A_n$, then $A_\nabla$
is derivable in QHC from $\prin(A_1)_\nabla,\dots,\prin(A_n)_\nabla$.
\end{theorem}

Let us note that (a) is only a meta-judgement, that is, it does
not claim that $A\turnstile A_\nabla$ in QHC, nor the converse
when $A$ is a formula of QH.
In fact these claims are false as we will see in \cite{M2}*{Remark \ref{g2:nabla-meta}}.

Of course, by \ref{McKinsey+Gentzen}(d,e,f) we may redefine
the $\nabla$-interpretation $A\mapsto A_\nabla$, without changing its effect, so as to prefix all
intuitionistic connectives and quantifiers of $A$ (not just $\lor$ and $\exists$)
and all atomic i-formulas by $\nabla$.
Alternatively, by \ref{Goedel+Kuroda}(a,b,c) we might redefine
the $\nabla$-interpretation $A\mapsto A_\nabla$, without changing its effect, so as to prefix
the entire formula $A$, if it represents a i-formula, by $\nabla$,
and {\it post}fix every intuitionistic $\to$ and $\forall$ by $\nabla$'s
(atomic subformulas are now kept intact).

\begin{proof}[Proof. (a)]
Let $A_\Box$ be the $\Box$-interpretation of $A$ regarded
as a formula of QHC, by identifying $\Box$ with $\wn\oc$.
Thus $A_\Box$ can be obtained from $A$ by first erasing every $\wn$ and
replacing every $\oc$ by $\wn\oc$, then prefixing all atomic i-formulas and all intuitionistic connectives 
and quantifiers by $\wn\oc$, and finally retyping all atomic i-formulas as c-formulas and replacing 
all intuitionistic connectives and quantifiers by the corresponding classical ones.
By \ref{box-int}, $\turnstile A$ implies
$\turnstile A_\Box$, and (since the classical provability translation of QH in QS4 is faithful) 
the converse holds when $A$ is a formula of QH.
Let $A_\Box'$ denote the formula of QHC obtained from $A$ by first erasing
every $\wn$ and replacing every $\oc$ by $\wn\oc$, then prefixing all atomic i-formulas and all occurrences of
$\ab$ and $\triv$ by $\wn$ and all other intuitionistic connectives and quantifiers by $\wn\oc$, and finally 
replacing all intuitionistic connectives and quantifiers except $\ab$ and $\triv$ by the corresponding 
classical ones.
Then $\turnstile A_\Box'$ implies $\turnstile A_\Box$ by
substituting the $\oc$-images of atomic c-formulas for the atomic
i-formulas of $A_\Box'$.
The converse implication follows by substituting $\wn$-images
of atomic i-formulas for the atomic c-formulas of $A_\Box$
and using $\turnstile \wn\oc\wn\alpha\tofrom\wn\alpha$.

On the other hand, as observed above, we may assume $A\mapsto A_\nabla$ to
prefix all intuitionistic connectives and quantifiers of $A$ (not just $\lor$ and
$\exists$) and all atomic i-formulas by $\oc\wn$.
Let $A_\nabla'$ denote $A_\nabla$ if $A$ is a c-formula, and $\wn A_\nabla$ if 
$A$ is an i-formula.
If $\Phi$ is an i-formula, then $\Phi_\nabla$ is an i-formula of the form $\oc\wn\Psi$.
In this case, we have $\Phi_\nabla\turnstile\wn\Phi_\nabla$ by ($\wn_\top$), and
conversely $\wn\Phi_\nabla\turnstile\Phi_\nabla$ since $\wn\Psi\turnstile\oc\wn\Psi$ by ($\oc_\top$).
Thus $\turnstile A_\nabla\iff A_\nabla'$ for any formula $A$.

Each judgement of QHC of the form $\turnstile A$ corresponds to a rooted tree whose root is labelled 
with $\turnstile$, whose leaves are labelled with the atomic subformulas of $A$ or with nullary connectives,
and whose other vertices are labelled with the unary and binary connectives and the quantifiers of $A$.
We can draw this tree on the plane so that every connective or quantifier is drawn above those in 
the subformulas that it applies to.
(Thus the root is at the top, and the leaves are in the bottom.)
Then $\wn$'s and $\oc$'s alternate along every path that is vertical
(in the sense that its projection to the vertical axis is a monotone function).
Hence $\wn$'s and $\oc$'s partition the tree into intuitionistic
fragments, bounded below by $\oc$'s or atomic i-formulas or $\ab$ or $\triv$, and above
by a $\wn$ or by the $\turnstile$; and classical fragments, bounded
below by $\wn$'s or atomic c-formulas or $\clbot$ or $\cltop$, and above by an $\oc$ or
by the $\turnstile$.

To analyze the difference between $A_\Box'$ and $A'_\nabla$,
we can use the tree of $A$ and write any prefix added to
a connective or quantifier of $A$ on the edge just above the vertex that
it labels, in respective order.
Then the difference is confined to the intuitionistic fragments of
the tree of $A$, and is that $A_\Box'$ has an extra $\wn$ (with respect to
$A_\nabla'$) just below each intuitionistic
connective and quantifier and a missing $\wn$ just above it.
Then we may push the extra $\wn$'s up using ($\wn_\land$),
($\wn_\lor$) and ($\wn_\exists$) as well as \ref{worst}(c,d), thus
obtaining that $\turnstile A'_\Box\tofrom A'_\nabla$.
(Alternatively, one can push $\oc$'s down, using ($\oc_\land$) and ($\oc_\forall$) as well as
\ref{move-oc-wn}(a) and \ref{worst}(a,b).)
\end{proof}

\begin{proof}[(b,c)]
By (a) all laws of QHC hold under the $\nabla$-interpretation.
The inference rules of QHC:
\formulas
$\dfrac{\alpha,\alpha\to\beta}{\beta}$, $\dfrac{\alpha(\tr x)}{\forall\tr x\,\alpha(\tr x)}$,
$\dfrac{p,p\to q}{q}$, $\dfrac{p(\tr x)}{\forall\tr x\,p(\tr x)}$, 
$\dfrac{\alpha}{\wn\alpha}$, $\dfrac{p}{\oc p}$
\metameta 
--- are easily seen to hold under the $\nabla$-interpretation, as they do not involve intuitionistic $\lor$ and $\exists$.

Finally, let us note that by an inductive argument based on
\ref{McKinsey+Gentzen}(d,e,f), $\turnstile\Phi_\nabla\tofrom\nabla\Phi_\nabla$ for any i-formula $\Phi$.
It follows that the second-order meta-specialization of the intuitionistic type holds under 
the $\nabla$-interpretation.
The other meta-rules hold under the $\nabla$-interpretation for trivial reasons.
\end{proof}

\subsection{$\neg\neg$-interpretation}\label{Tarski-Kolmogorov1}
The classical $\neg\neg$-translation of QC in QH (see \cite{M0}*{\S\ref{int:negneg}}) extends to 
the following syntactic {\it $\neg\neg$-interpretation} of QHC in QH, denoted by
$A\mapsto A_{\neg\neg}$:
\begin{itemize}
\item Atomic i-formulas and intuitionistic connectives remain unchanged;
\item Atomic c-formulas are re-typed as atomic i-formulas and are prefixed by $\neg\neg$;
\item Classical $\land$, $\to$ and $\forall$ become intuitionistic, and $\clbot$, $\cltop$ are replaced by 
$\ab$, $\triv$;
\item Classical $\lor$ and $\exists$ become intuitionistic and are prefixed by $\neg\neg$;
\item $\oc$ is erased, and $\wn$ is replaced by $\neg\neg$.
\end{itemize}

\formulas
Indeed, let us write out the images of the laws and inference rules in \ref{simplified}
under the $\neg\neg$-interpretation:

($\oc_\top$) $p\,/\,\oc p$ becomes $\neg\neg\pi\,/\,\neg\neg\pi$;

($\wn_\top$) $\alpha\,/\,\wn\alpha$ becomes $\alpha\,/\,\neg\neg\alpha$;

($\oc\wn$) $\prin\alpha\to\nabla\alpha$ becomes $\prin\alpha\to\neg\neg\alpha$;

($\wn\oc$) $\prin\wn\oc p\to p$ becomes $\prin\neg\neg\neg\neg\pi\to\neg\neg\pi$;

($\oc_\bot$) $\prin\oc\clbot\to\clbot$ becomes $\prin\ab\to\ab$;

($\octo$) $\prin\oc(p\to q)\to(\oc p\to\oc q)$ becomes
$\prin(\neg\neg\pi\to\neg\neg\rho)\to(\neg\neg\pi\to\neg\neg\rho)$;

($\wnto$) $\prin\wn(\alpha\to\beta)\to(\wn\alpha\to\wn\beta)$ becomes 
$\prin\neg\neg(\alpha\to\beta)\to(\neg\neg\alpha\to\neg\neg\beta)$.

\noindent
The resulting formulas are easily derivable in intuitionistic logic, including the last one 
(see \cite{M0}*{\S\ref{int:tautologies}, (\ref{int:neg-neg-imp})}).
\metameta

Intuitionistic laws and inference rules of QHC hold under the $\neg\neg$-interpretation
since it does nothing to intuitionistic connectives and quantifiers and to atomic i-formulas.
Classical laws and inference rules of QHC hold under the $\neg\neg$-interpretation
since the restriction of the $\neg\neg$-interpretation to QC is known
to be an interpretation (see \cite{M0}*{\S\ref{int:negneg}}).

Finally, let us note that by an inductive argument based on
\cite{M0}*{\ref{int:Goedel-Gentzen}}, $\turnstile F_{\neg\neg}\tofrom\neg\neg F_{\neg\neg}$
for every c-formula $F$.
It follows that the second-order meta-specialization of the classical type holds under 
the $\neg\neg$-interpretation.
The other meta-rules hold under the $\neg\neg$-interpretation for trivial reasons.

We have proved

\begin{theorem} (a) If $A_1,\dots,A_n\turnstile _{QHC} A$, then
$(A_1)_{\neg\neg},\dots,(A_n)_{\neg\neg}\turnstile _{QH} A_{\neg\neg}$.

(b) If a formula $A$ is derivable in QHC from $\prin A_1,\dots,\prin A_n$, then $A_{\neg\neg}$
is derivable in QH from $\prin(A_1)_{\neg\neg},\dots,\prin(A_n)_{\neg\neg}$.
\end{theorem}

Of course, by \cite{M0}*{\ref{int:Goedel-Gentzen}}, all of the classical connectives and quantifiers
(and not only $\lor$ and $\exists$) could be prefixed by a $\neg\neg$
in the definition of the $\neg\neg$-interpretation.
Consequently, by \cite{M0}*{\ref{int:Kuroda}} one could formulate
the $\neg\neg$-interpretation as an extension of Kuroda's translation:
prefix by a $\neg\neg$ the entire formula if it is a c-formula,
{\it post}fix by $\neg\neg$'s every classical $\forall$, replace every
$\oc$ by a $\neg\neg$, and erase all $\wn$'s.
Like before, all classical connectives and quantifiers become intuitionistic, and
all atomic c-formulas are re-typed as atomic i-formulas.

As in \cite{M0}, we also get an ``essentially local'' version of the $\neg\neg$-interpretation:
\begin{itemize}
\item postfix by a $\neg\neg$ every classical $\forall$ and $\land$;
\item prefix by a $\neg\neg$ every classical $\exists$ and $\lor$;
\item postfix by a $\neg\neg$ every $\oc$ that is followed by a $\wn$ or by an atomic c-formula;
\item prefix by a $\neg\neg$ the entire formula if it is an atomic c-formula or starts with $\wn$;
\item now erase all $\wn$'s and $\oc$'s;
\item all classical connectives and quantifiers become intuitionistic, and
all atomic c-formulas are re-typed as atomic i-formulas.
\end{itemize}

Since the restriction of the $\neg\neg$-interpretation to QH is
the identity, we obtain

\begin{theorem}\label{conservative2}
The QHC calculus is a strongly conservative extension of QH.
\end{theorem}

\subsection{$\Diamond$-interpretation}\label{diamond}
By Theorem \ref{conservative2},
the $\neg\neg$-interpretation of QHC in QH can be regarded as
an interpretation of QHC in itself which does not preserve the types of
formulas.
In this sense it can be improved, so as to preserve the typing.
This results in the following {\it $\Diamond$-interpretation} of QHC in
itself, which restricts to an unintended embedding of QC:

\begin{theorem} \label{Diamond-interpretation}
If $A$ is a formula of QHC, let $A_\Diamond$ be the formula
of QHC obtained from $A$ by prefixing
\begin{itemize}
\item classical $\to$ and $\forall$ by $\Box$; and
\item atomic c-formulas, $\wn$, and classical $\lor$ and $\exists$ by $\Box\Diamond$,
\end{itemize}
where $\Box=\wn\oc$ and $\Diamond=\neg\Box\neg$.
Then 

(a) $\turnstile A$ implies $\turnstile A_\Diamond$, and the converse holds when $A$ is a formula of QC 
or (trivially) of QH;

(b) if $A_1,\dots,A_n\turnstile A$, then $(A_1)_\Diamond,\dots,(A_n)_\Diamond\turnstile A_\Diamond$;

(c) if a formula $A$ is derivable in QHC from $\prin A_1,\dots,\prin A_n$, then $A_\Diamond$
is derivable in QHC from $\prin(A_1)_\Diamond,\dots,\prin(A_n)_\Diamond$.
\end{theorem}

The conclusion of (a) is to be read as a meta-judgement, that is, it does
not claim that $A\turnstile A_\Diamond$ in QHC, nor the converse
when $A$ is a formula of QC.
In fact the former claim is false as we will see in \cite{M2}*{Remark \ref{g2:Diamond-meta}}.

The proof of Theorem \ref{Diamond-interpretation} is similar to that
of Theorem \ref{nabla-interpretation}, using additionally diagram ($*$) below.
Let us only note that the combination $\Box\Diamond$ arises by applying
the $\Box$-interpretation (in the prefixing version) to a $\neg\neg$.
When this $\neg\neg$ is in front of an atomic c-formula, that atomic
c-formula would have to been prefixed by $\Box\Diamond\Box$; but the last $\Box$
is easily seen to be redundant.

By starting from different versions of the $\neg\neg$-interpretation,
and applying different versions of the $\Box$-interpretation, one gets
a few equivalent forms of the $\Diamond$-translation.
For example, by starting from Kolmogorov's original form of the $\neg\neg$-translation,
we get the following succinct version of the $\Diamond$-interpretation:
\begin{itemize}
\item Prefix all classical connectives and quantifiers, all atomic c-formulas, and all $\wn$'s by $\Box\Diamond$.
\end{itemize}
This interpretation extends Fitting's translation of classical logic in QS4 \cite{Fi}.

On the other hand, by starting with the ``essentially local'' form of the $\neg\neg$-interpretation,
and applying the prefix form of the $\Box$-translation, we get the following version
of the $\Diamond$-interpretation: Prefix classical $\to$ and $\forall$, and
atomic c-formulas by a $\Box$; prefix classical $\lor$ and $\exists$ by a
$\Box\Diamond$; postfix classical $\forall$ and $\land$ by a $\Box\Diamond$;
prefix  by a $\Diamond$ every $\oc$ that is followed by a $\wn$ or by an atomic c-formula;
and if the entire formula is an atomic c-formula or starts with $\wn$, prefix it by a $\Diamond$.
Now atomic c-formulas do not really need to be prefixed by $\Box$'es, since they
are anyway effectively prefixed by double negations in the form $\Box\Diamond$,
which clearly suffices.
Next, the prefix $\Box\Diamond$ of classical $\lor$'s and $\exists$'s can be reduced
to a mere $\Diamond$ by the price of {\it post}fixing classical $\to$'s, $\lor$'s
and $\exists$'s by $\Box$'es.
Finally, by \cite{M0}*{\ref{int:QS4-McKinsey}(a)}, it does not hurt to also prefix
classical $\land$'s by $\Box$'es; and given that, by \cite{M0}*{\ref{int:QS4-Goedel}},
the postfix $\Box\Diamond$ of classical $\forall$'s and $\land$'s can be reduced to a mere $\Diamond$.

To summarize, we get the following ``essentially local'' form of the $\Diamond$-interpretation:
\begin{itemize}
\item Prefix and postfix every classical $\to$ by a $\Box$;
\item prefix classical $\forall$'s and $\land$'s by $\Box$'es, and postfix them by $\Diamond$'s;
\item prefix classical $\lor$'s and $\exists$'s by $\Diamond$'s, and postfix them by $\Box$'es;
\item postfix by a $\Diamond$ every $\oc$ that is followed by a $\wn$ or by an atomic c-formula;
\item prefix by a $\Diamond$ the entire formula if it is an atomic c-formula or starts with $\wn$.
\end{itemize}

Using the usual identities, this can be further reformulated in a more economical way in terms of
$\oc$ and $\wn$:

\begin{itemize}
\item Prefix every classical $\to$ by a $\wn$ and postfix it by an $\oc$;
\item prefix every classical $\forall$ and $\land$ by a $\wn$, and postfix it by $\neg\oc\neg$;
\item prefix every classical $\lor$ and $\exists$ by $\neg\wn\neg$, and postfix it by an $\oc$;
\item replace every $\nabla$ by a $\neg\neg$;
\item replace by $\neg\oc\neg$ every $\oc$ that is followed by an atomic c-formula;
\item if the formula starts with $\wn$, replace that $\wn$ by $\neg\wn\neg$;
\item if the entire formula is an atomic c-formula, prefix it by a $\Diamond$.
\end{itemize}

\formulas
This can be regarded as an alternative form of Kolmogorov's original $\neg\neg$-translation, since it has
the effect of expressing all classical connectives and quantifiers in terms of the intuitionistic ones
along with $\wn$, $\oc$ and the classical negation. 
In particular,
\[\exists\tr x\,p(\tr x)\quad\text{is translated as}\quad \neg\wn\neg\exists\tr x\oc p(\tr x).\]
In words, there exists an $\tr x$ such that $p(\tr x)$ if and only if it is impossible to derive
a contradiction from a construction of an $\tr x$ along with a proof of $p(\tr x)$.
\metameta

\subsection{Applications to QH4}
The $\nabla$-interpretation is easily seen to lift to the interpretation of QH4 in itself described
by Aczel \cite{Ac}:

\[
\begin{CD}
QH4@>\text{Aczel's interpretation}>>QH4\\
@V{\nabla=\oc\wn}VV@VV{\nabla=\oc\wn}V\\
QHC@>\text{$\nabla$-interpretation}>>QHC.
\end{CD}
\]
\smallskip

The vertical arrows of this diagram commute with the inclusions of QH into QH4 and into QHC,
so they are faithful on QH (since QHC is a conservative extension of QH).
On the other hand, the $\nabla$-interpretation was shown to be faithful on QH, so we conclude
that Aczel's interpretation restricts to an unintended embedding of QH into QH4, which we will call
the {\it $\nabla$-translation}.

Since $\fm{\turnstile \nabla\neg\neg\alpha\tofrom\neg\neg\alpha}$ not only in QHC, but also in QH4 
(see \cite{Ac}), we have the commutative diagram

\[
\begin{CD}
QC@>\text{$\neg\neg$-translation}>>QH\\
@V\text{$\neg\neg$-translation}VV@VV\text{inclusion}V\\
QH@>\text{$\nabla$-translation}>>QH4.
\end{CD}
\tag{$*$}
\]
\smallskip

The $\neg\neg$-interpretation of QH4 in QH replaces every occurrence of $\nabla$ by $\neg\neg$.
On the other hand, the $\nabla$-translation of QH into QH4 and
the $\neg\neg$-translation of QC into QH are defined using similar
formulas $A_\nabla$ and $A_{\neg\neg}$, which, as discussed
above, can both be written out by prefixing {\it all} connectives, quantifiers
and atomic subformulas with either $\nabla$ or $\neg\neg$.
Hence $A_\nabla$ and $A_{\neg\neg}$ become equivalent
upon substituting $\nabla$ by $\neg\neg$, and we get
the following commutative diagram.

\[
\begin{CD}
QH@>\text{$\nabla$-translation}>>QH4\\
@V\fm{\prin\alpha\lor\neg\alpha} VV@VV\nabla\mapsto\neg\neg V\\
QC@>\text{$\neg\neg$-translation}>>QH.
\end{CD}
\]
\smallskip

By \ref{pushout formula}, which can in fact be proved in QH4 and not just in QHC (we leave this for
the reader to check), this is a pushout diagram, that is, the $\neg\neg$-translation of QC into QH does not
factor through any logic obtained by adding a set of laws to QH4 that are collectively strictly
weaker than $\fm{\prin\nabla\alpha\tofrom\neg\neg\alpha}$.
In this sense, the classical $\neg\neg$-translation of QC into QH cannot be improved.

\section{Discussion}

\subsection{Knowledge-that vs.\ knowledge-how}\label{philosophy}

Kolmogorov has summarized his philosophical views on intuitionism in his foreword%
\footnote{Which must have been addressed in part to the Soviet censor, as it included
the obligatory denunciation of subjective idealism.
This could well have implications for the wording and emphasis chosen, but hardly for the sincerity
of Kolmogorov's words (as one can judge from his published correspondence with Alexandrov
and from the transcript of Luzin's trial).}
to a 1936 translation of Heyting's book \cite{He2} (translated from Russian):
\smallskip
\begin{center}
\parbox{14.7cm}{\small
``We cannot agree with intuitionists when they say that mathematical
objects are products of constructive activity of our spirit.
For us, mathematical objects are abstractions of actually existing
forms of reality, which is independent of our spirit.
But we know how essential in mathematics is, in addition to pure proof of
theoretical propositions, constructive solution of posed problems.
This second, constructive side of mathematics does not eclipse
for us its first and foremost side: the cognitive one.
However, the laws of mathematical construction, discovered by Brouwer
and systematized by Heyting under the guise of a new intuitionistic
logic, keep their fundamental importance for us, in their present
understanding.''
}
\end{center}
\noindent
This is quite in line with a passage from Kolmogorov's 1929 survey \cite{Kol29}:

\begin{center}
\parbox{14.7cm}{\small
``We could distinguish two sides in this concept of mathematics. On the one side, there
are theories postulating the existence of infinite systems of objects satisfying certain axioms and
formally deriving from the axioms the properties of the system being studied. On the other side,
construction of the corresponding objects, based either on positive integers or on some other resource
of elementary objects, is also recognized as necessary. Experience of the last years shows that
no stable balance was attained between these two sides. The standpoints that came to light in recent times
may be roughly formulated as follows. Hilbert proposed to keep only the former, formal part of mathematics,
while having set us free, by means of his theory of consistency, from the necessity to construct.
On the contrary, Brouwer values mainly the constructive part, but thinks that construction is unable to
give us the ultimate existence of infinite collections that is needed for a free use of the ways of reasoning
that have became common to mathematics; and therefore he demands a radical revision of
the methods of a mathematical proof.

The emergence of these extreme viewpoints is explained by the fact that joining of
the two sides of the set-theoretic mathematics has led to great difficulties and even
contradictions.'' (There follows a discussion of Russell's paradox, Weyl's predicativist restrictions, and
non-measurable sets.)
}
\end{center}
\noindent
The two sides of mathematics referred to by Kolmogorov can be seen as representing
two modes of knowledge (including formalized mathematical knowledge, but also keeping in mind subjects
such as common knowledge and collective intelligence):
\begin{itemize}
\item {\it knowledge-that} (or knowledge of truths), and
\item {\it knowledge-how} (or knowledge of methods).
\end{itemize}
This dichotomy is also noted, from a slightly different perspective, in \cite{F-D}:

\smallskip
\begin{center}
\parbox{14.7cm}{\small
``In contrast with the structural (platonistic) point of view, intuitionistic mathematics focuses
primarily on the {\it subject} (the creative mathematician) and his ability to perform certain
mathematical operations by applying his previously designed constructions ({\it knowing how}).
Hence a notion such as `proof', which refers to the successful completion of a human action, appears
to be more suitable than that of `truth'.
\newline On the other hand, classical mathematics focuses essentially on the {\it object}: eternal
pre-existing mathematical structures ({\it knowing that}); and for this reason, the notion of `truth',
with its prominent descriptive untensed character, is more appropriate.''
}
\end{center}
\smallskip

There is, however, hardly any connection with the distinction made in philosophy between
``knowledge how'' and ``knowledge that''  --- in the tradition originating with G. Ryle, whose
``knowledge how'' is an unconscious, non-articulable ability.%
\footnote{Martin-L\"of argued of ``knowledge how in Ryle's terminology'' that
``the distinction between knowledge how and knowledge that evaporates on the intuitionistic analysis
of the notion of truth.'' \cite{ML96}*{p.~36}.}
The same can be said of procedural vs.\ declarative knowledge of cognitive psychology.
Somewhat closer to our concern here are the distinction between declarative and imperative programming
languages, and another distinction made in philosophy, starting with B. Russell:
knowledge of objects (including mathematical objects) ``by acquaintance'' vs.\ ``by description''.
Hilbert's distinction between formal mathematics (subject only to freedom from contradiction) and
intuitively justifiable, ``finitistic'' methods (especially as reinterpreted in \cite{We}) is also
to the point.
Mathematically most relevant is, of course, Lawvere's adjunction between the Formal and the Conceptual.

Our connectives $\oc$ and $\wn$ amount to two ``conversion'' operators between the two modes of knowledge.
These give rise to compound types of knowledge:
\begin{enumerate}
\item knowledge-that {\it there-exists} a knowledge-how (or knowledge of the possible);
\item knowledge-how {\it to-acquire} the knowledge-that (or knowledge of reasons).
\end{enumerate}
Here knowledge-how to-acquire the knowledge-that some mathematical assertion is true means, of course,
knowledge-how to prove that assertion (cf.\ \cite{ML96}*{pp.\ 28--29}).
In general, (2) could be dubbed ``knowledge-why'' or even ``understanding''.

Now, (1) occurs most distinctively whenever one applies a non-constructive existence theorem.
For instance, for those who feel at home with ZFC, presumably one is supposed to have
the knowledge-that there-exists a knowledge-how to well-order the reals (without being aware of any specific
well-ordering).
For those who feel more at home with constructive mathematics, a more down-to-earth example is provided by
constructive proof-checkers, such as Coq.
If you know that Coq works correctly on your computer%
\footnote{For instance, if you have manually verified the code of its rather small kernel (which in turn
verifies all the needed extensions), and if you believe that one can neglect potential bugs in
the operating system and in the design of
the microprocessor, as well as possibilities of malfunctioning due to a manufacturing defect, heat or
irregularities of power supply (or just because of a microscopic meteorite).}
and you acquired a file with a fully Coq-formalized proof of, say, the Four Color Theorem,%
\footnote{Without any tricks smuggling in the law of excluded middle as in \cite{Gu}}
then by running Coq to certify this proof you
would presumably acquire the knowledge-that there-exists a knowledge-how to color any given planar map in
four colors (without getting any clue how to do the actual coloring).

\subsection{Understanding historic writings}

\subsubsection{Orlov--Heyting interpretation}\label{letters1}
G\"odel's provability translation, as well as his sketch of a proof-relevant S4 that we relied upon
in motivating our formulation of QHC (see \S\ref{deductive}) were anticipated by informal
provability interpretations of intuitionistic logic in the papers
by Heyting \cite{He0}, \cite{He1} and, independently,
Orlov \cite{Or}*{\S\S6,7} (see also \cite{D92} for a discussion
of Orlov's work in English).
We will now briefly review these papers, which will also prepare us for
a discussion of Kolmogorov's letter to Heyting \cite{Kol2}.

All three papers focus mainly on a pair of operators, which we will denote by $\Plus$ and $\Neg$ throughout,
following \cite{He0} and \cite{Kol2}.
(Orlov \cite{Or} writes $\Phi$ and $X$; and in his second paper \cite{He1}, Heyting writes $\Plus$ and $\neg$. 
For consistency, we will alter these to $\Plus$ and $\Neg$, respectively, when quoting from these papers.)
The meaning of these operators will be discussed in a moment.

Heyting and Kolmogorov also use the symbol $\turnstile$ to mean something quite different from both 
the modern meaning of this symbol and its original meaning as used by Frege and Russell.
According to Heyitng \cite{He0}:
\smallskip
\begin{center}
\parbox{14.7cm}{\small
``To satisfy the intuitionistic demands, the assertion must be the observation of an empirical fact, 
that is, of the realization of the expectation expressed by the proposition $p$.
Here, then, is the {\it Brouwerian assertion} of $p$: {\it It is known how to prove $p$.} 
We will denote this by $\turnstile p$. 
The words `to prove' must be taken in the sense of `to prove by construction'.''
}
\end{center}
\medskip
It is not really clear to the author exactly what this may mean from the viewpoint of classical meta-logic.
But if letters used for unknown propositions (such as $p$ in Heyting's words above) are understood as 
propositional variables (or meta-variables for propositional variables), then, with an appropriate 
interpretation of $\Plus$ and $\Neg$ (discussed below), $\turnstile$ may be read in a usual way, as 
asserting provability in the modal logic S4.
With this in mind, we will follow Heyting et al.\ in using lowercase letters for propositions in the
present section (in contrast to the notation elsewhere in the present series of papers).

If $p$ is a proposition, both Heyting and Orlov interpret $+p$ as
{\it $p$ is provable}.
Orlov states unambiguously that $\Plus$ is constrained precisely by what
turns out to be the modal axioms of S4;%
\footnote{Orlov's own system of axioms is weaker than S4 in that it is
based not on classical logic, but on a weaker system now known as
relevant logic, which satisfies $\turnstile\fm{p\tofrom\neg\neg p}$ but neither
of $\turnstile\fm{\neg p\to(p\to q)}$, $\turnstile\fm{ q\to(p\to q)}$, $\turnstile\fm{ p\lor\neg p}$
(see \cite{D92}).
Orlov erroneously believed that the use of full classical logic would
trivialize the $\Plus$ operator.}
Heyting says only that ``A logic that would treat properties of
the function $\Plus$ would [...] be purely hypothetical; [...] one cannot
ask [the intuitionistic mathematicians] to develop this logic'' \cite{He0}.

Heyting interprets $\Neg p$ as ``$p$ implies a contradiction'' and
calls $\Neg$ ``the Brouwerian negation''; whereas Orlov says that it
``has the same meaning'' as Brouwer's notion of ``absurdity'' of
a proposition in \cite{Br25}.
At the same time, Orlov is able to identify $\Neg p$ as $\Plus\neg p$,
where $\neg$ is the classical negation; in this connection, Heyting
only says: ``the negation of a proposition always refers to a proof
procedure which leads to the contradiction, even if the original
proposition mentions no proof procedure'' \cite{He1}.

In his second paper \cite{He1}, Heyting also states:
\smallskip
\begin{center}
\parbox{14.7cm}{\small
``[I]ntuitionist logic, insofar as it has been developed up to now without
using the function $\Plus$, must be understood [... in the sense of] treating
only propositions of the form `$p$ is provable' or, to put it another way,
by regarding every intention as having the intention of a construction
for its fulfillment added to it.''
}
\end{center}
\medskip
In practical terms this means, in particular, that Brouwer's theorem
on triple absurdity \cite{Br25} should be interpreted as
$\turnstile \Neg\Neg\Neg\Plus p\tofrom\Neg\Plus p$.
Thus we consider Heyting's earlier claim in \cite{He0} that
``Mr. Brouwer has proved that $\Neg\Neg\Neg p$ is identical to $\Neg p$''
to be in error, as pointed out essentially by Heyting himself.
This fully agrees with Orlov's independent analysis, which contains
valid proofs of  $\turnstile\Plus p\to\Neg\Neg\Plus p$ and
$\turnstile \Neg\Neg\Neg\Plus p\tofrom\Neg\Plus p$ and informal
arguments that  $\turnstile p\not\to\Neg\Neg p$ and $\turnstile\Neg\Neg\Neg p\not\to\Neg p$.
Orlov also supported these judgements by an analysis
of Brouwer's writings:
\smallskip
\begin{center}
\parbox{14.7cm}{\small
``Brouwer often resorts to the following method of defining notions:
`We call a real number $g$ rational if two whole numbers $p$ and $q$
can be specified so that $g=p/q$; and irrational if one can make
the assumption of the rationality of $g$ to lead to absurdity.'
[\cite{Br25}]

Here it is evident that a rational number is defined via provability
of the existence of the two integers, in other words, by a function
of the form $\Plus a$.
If the assumption that the existence of $p$ and $q$ is provable
leads to absurdity, then $g$ is irrational.
Therefore, irrationality is defined by means of $\Neg\Plus a$.''
}
\end{center}
\medskip
Nevertheless, Orlov was only partially aware of Heyting's principle
quoted above, for he interpreted the intuitionistic understanding of
the principle of excluded middle as $\Plus p\lor\Neg p$, rather than
$\Plus p\lor\Neg\Plus p$.

Apart from these oddities, Heyting's both papers and Orlov's paper seem
to be compatible with each other and with G\"odel's provability
translation.
In particular, both Heyting and Orlov mention the equivalence of
$\Plus\Plus p$ with $\Plus p$ and of $\Plus\Neg p$ with $\Neg p$; and
of the judgements $\turnstile\Plus p$ and $\turnstile p$.
We should mention, however, yet another oddity found in Heyting's letter
to Freudenthal, where he first gave an interpretation of the intuitionistic
negation (see \cite{vA}):

\smallskip
\begin{center}
\parbox{14.7cm}{\small
``I believe that also $a\to b$, like the negation, should refer to a proof
procedure: `I possess a construction that derives from every proof of $a$
a proof of $b$'. In the following, I will keep to this interpretation.
There is therefore no difference between $a\to b$ and $+a\to+b$.
}
\end{center}
\medskip

Note that this is at odds already with Heyting's own distinction
between $\Neg a$ and $\Neg+a$, which he was clear about in \cite{He0}.

We will thus refer to the {\it standard interpretation} of $\Plus$ and $\Neg$, where ``propositions'' 
are formalized as formulas of S4, $\Plus$ is identified with the modality $\Box$ of S4, and $\Neg$ is 
regarded as an abbreviation for $\Box\neg$.
The latter abbreviation is, in fact, very convenient also from
a technical viewpoint, for it gives a more intelligible form to judgements
that correspond to basic properties of a subset of a topological space
under the topological interpretation of S4:
\begin{enumerate}
\item $\turnstile\Plus F$ (or $\turnstile F$): Entire space
\item $\turnstile\Neg F$ (or $\turnstile\neg F$): Empty set
\item $\turnstile\Neg\Plus F$ (or $\turnstile\neg\Box F$): Boundary set
($\iff$ interior is empty $\iff$ complement is dense)\!\!
\item $\turnstile\Neg\Neg F$ (or $\turnstile\neg\Box\neg F$): Dense set
($\iff$ closure is the entire space)
\item $\turnstile\Neg\Neg\Plus F$ (or $\turnstile\neg\Box\neg\Box F$):
Complement is nowhere dense ($\iff$ interior is dense)
\item $\turnstile\Neg\Neg\Neg F$ (or $\turnstile\neg\Box\neg\Box\neg F$):
Nowhere dense set ($\iff$ closure is a boundary set)
\end{enumerate}
Note that by Brouwer's theorem that $\turnstile \Neg\Neg\Neg\Plus F\tofrom\Neg\Plus F$,
this list cannot be continued any further.
It is immediate from $\turnstile\Box F\to F$ that
\begin{itemize}
\item (1) implies (5), which in turn implies (4); and
\item (2) implies (6), which in turn implies (3).
\end{itemize}
Also, it is immediate from the necessitation rule that the following pairs of judgements are contradictory: 
\begin{itemize}
\item (6) contradicts (4); 
\item (4) contradicts (2); 
\item (2) contradicts (1); 
\item (1) contradicts (3); 
\item (3) contradicts (5).
\end{itemize}

Heyting \cite{He0} overlooked only the last entry of the list (1)--(6).
Accordingly, in his discussion of possible combinations of these judgements he missed precisely those
that involve (6).
His ``possible combinations'' consist of judgements that neither imply nor contradict one another.
It is easy to check that there are just nine of them: the empty combination; the six singleton combinations;
and two pairs: (3)+(4) and (5)+(6).
Of these, only (1), (2), (3)+(4) and (5)+(6) are ``definitive'' in Heyting's terminology, that is,
cannot be extended to a larger combination.

\subsubsection{Kolmogorov's letters to Heyting}\label{letters2}
The combination (3)+(4) was discussed by Heyting in detail \cite{He0}:

\smallskip
\begin{center}
\parbox{14.7cm}{\small
``[L]et us consider the proposition `Every even number is a sum of two
primes' (Goldbach's conjecture).
Then $p$ means simply that in taking an even number at random, one expects
to be able to find two primes of which it is the sum.
(This possibility is decided after a finite number of attempts.)
$+p$ on the contrary requires a construction that gives us this
decomposition for all even numbers at once. [...]
In order to be able to assert $\turnstile\Neg\Plus p$, it suffices to
reduce to a contradiction the assumption that one can find a construction
proving $p$; by that one will not yet have proved that the assumption
$p$ itself implies a contradiction.
If we appeal to the example of Goldbach's conjecture, we find:
$\turnstile\Neg\Plus p$ means that one will never be able to find a rule
that effects in advance the decomposition of all even numbers; this does not
mean that there is a contradiction
when one supposes that in taking an even number at random, one will
always be able to divide it into two prime numbers.
It is even conceivable that it could one day be proved that this last
supposition cannot lead to a contradiction; then one would have
at the same time $\turnstile\Neg\Plus p$ and $\turnstile\Neg\Neg p$.
One should abandon every hope of ever settling the question; the problem
would be unresolvable.''
}
\end{center}
\smallskip
Heyting's second paper \cite{He1} contains a virtually identical
discussion but with a different choice of $p$, namely, the one asserting
that a given rational number lies within every interval with rational
endpoints that contains Euler's constant $C$.

Heyting's claims are confirmed rigorously in the standard interpretation,
since there exist dense boundary sets (for instance, $\Q$ viewed as
a subset of $\R$).
This is in contrast with Kolmogorov's claims in his first letter to
Heyting \cite{Kol2}:

\smallskip
\begin{center}
\parbox{14.7cm}{\small
``1. You consider as an example (in [\cite{He0}]) the proposition
`Every even number is a sum of two primes'.
But it is known that the formula $\turnstile\Neg\Neg p\to p$ is true in
this case from either classical or intuitionistic viewpoint.
If one asserts $\turnstile\Neg\Neg p$, then automatically there is
`a construction, which gives us this decomposition for all even numbers
at once'. Hence $\turnstile\Neg\Neg p\to \Plus p$, and the case
$\turnstile\Neg\Neg p\land\Neg\Plus p$ is impossible.

2. It seems to me that the point is not in a defect of this particular
example.
Each `proposition' in your framework belongs, in my view, to one of
two sorts:
\begin{itemize}
\item[($\alpha$)] $p$ expresses hope [l'esperance] that in prescribed
circumstances, a certain experiment will always produce a specified
result.
(For example, that an attempt to represent an even number $n$
as a sum of two primes will succeed upon exhausting all pairs
$(p,q)$, $p<n$, $q<n$.)
Of course, every `experiment' must be realizable by a finite
number of deterministic operations.
\item[($\beta$)] $p$ expresses the intention to find a construction.
\end{itemize}

3. We agree that in the case ($\beta$), the difference between $p$ and
$\Plus p$ is not essential, but the proposition $\Neg\Neg p\to p$
should not be regarded as evident.
In the first case ($\alpha$), on the contrary, $p$ and $\Plus p$ have
distinct meanings, but we have $\turnstile\Neg\Neg p\to p$ and
$\turnstile\Neg\Neg p\to \Plus p$.
This is why $\turnstile \Neg\Neg p\land\Neg\Plus p$ is always impossible,
both in the case ($\alpha$) and in the case ($\beta$).

4. I prefer to keep the name {\it proposition} (Aussage) only for
propositions of type ($\alpha$) and to call ``propositions'' of
type ($\beta$) simply {\it problems} (Aufgaben). Associated to
a proposition $p$ are the problems $\Neg p$ (to derive
contradiction from $p$) and $\Plus p$ (to prove $p$).''
}
\end{center}
\smallskip

Kolmogorov insists that every proposition $p$ of type ($\alpha$)
satisfies $\turnstile\Neg\Neg p\to\Plus p$, apparently because $p$ comes
endowed with a constructive procedure of verification of the validity
of every particular instance of $p$; from $\Neg\Neg p$ we infer that this
procedure actually returns a positive result on all inputs; thus it
yields a constructive proof of $p$.
This applies to the example of $p$ cited by Kolmogorov, {\it Every even
number is a sum of two primes}, since it can be verified constructively
whether a given specific number $n$ is a sum of two primes (by exhausting
all primes $<n$).
The same applies to the other example of Heyting, and in general to
every proposition of the form
$\forall x_1\dots\forall x_n\,q(x_1,\dots,x_n)$, where the validity of
$q$ is verifiable by a finite procedure.%
\footnote{These so-called $\Pi^0_1$ propositions are sometimes claimed to be precisely
the propositions accessible to Hilbert's finitistic reasoning (see \cite{Goe0}*{p.\ 191}).}

However, already the classical negation $\neg p$ of such a proposition
$p$, for instance, the proposition {\it There exists an even number
that is not a sum of two primes}, is presumably neither of
type ($\alpha$)
nor of type ($\beta$) --- since it does not assert that such
a number can be constructed explicitly.
(In contrast, $\Neg p$ must be of type ($\beta$), according to
Kolmogorov's (4).)
In his second letter to Heyting, Kolmogorov himself speaks of
propositions that are neither of type ($\alpha$) nor of type
($\beta$):

\smallskip
\begin{center}
\parbox{14.7cm}{\small
``In the meantime, I have thought about your example of the proposition
`For all $i$ we have $a_i<b_i$'.
Let, in general, $x$ be a variable and $P(x)$ a problem depending on $x$.
The `hope' [Hoffnung] to find for each $x$ a solution of the problem
$P(x)$ is neither a problem nor a proposition in my terminology.
It would be very interesting to know if with this hope you associate
a positive expectation [Erwartung] that for each $x$ the problem $P(x)$
will {\it really be solved} (by whom and when)?
If this expectation is not intended, then I am afraid that we will
arrive at the naive non-intuitionistic understanding of the statement
`$P(x)$ is solvable for each $x$'.''
}
\end{center}
\smallskip

If the $a_i$ and $b_i$ are assumed to be real numbers (Heyting's reply
to Kolmogorov's first letter, which would clarify this matter, is not
available; see, however, a fragment of Heyting's letter to Becker below)
then the proposition $a_i<b_i$ amounts to an existentially quantified
proposition about rational numbers (or integers).
In this case, Kolmogorov's statement `$P(x)$ is solvable for each $x$',
where the problem $P(x)$ is instantiated as ``Prove that $a_x<b_x$''
will be of the form $\forall x?!\exists y\, q(x,y)$, to use
the notation of QHC.
As observed by Troelstra \cite{Tr90},
\smallskip
\begin{center}
\parbox{14.7cm}{\small
``In the second letter [to Heyting] Kolmogorov observes that
the distinction between
`$P(x)$ can be solved for each $x$' and `there is a uniform method for
solving $P(x)$ for each $x$' is non-intuitionistic [that is, ``does not
fit into an intuitionistic point of view'']; the point was accepted by
Heyting, as the fragment of his letter to Becker, reproduced above,
shows.''
}
\end{center}
\smallskip
Here is the relevant part of the said fragment of Heyting's letter to
Becker \cite{Tr90} (translated from German):
\smallskip
\begin{center}
\parbox{14.7cm}{\small
``Another matter is that the application of my logic is restricted to
constructive questions.
What I mean by this may be illuminated by the following example.
Let two sequences of real numbers $\{a_i\}$ and $\{b_i\}$ be given.
The proposition `For each $i$, $a_i=b_i$' admits two interpretations.
\begin{itemize}
\item[a)] It can mean the problem of finding a general proof that
upon the choice of a particular index $i$ specializes to a proof of $a_i = b_i$;
\item[b)] one can understand by it the expectation [Erwartung] that if one
keeps choosing an index $i$ arbitrarily, he will succeed in proving $a_i = b_i$ every time.''
\end{itemize}
The difference is clear if one applies the negation to a) and b).
It is conceivable that the assumption of a proof as requested in a)
could be proved contradictory, without this contradiction affecting the assumption of b).
My logic applies if each proposition is understood as in a); the logic of the non-constructive
expectations b) would be much more involved; I do not consider its development to be very fruitful.''
}
\end{center}
\smallskip

\subsubsection{Interpreting Kolmogorov's letters}
To summarize our reading of the quoted writings, Heyting is right in that one can have at
the same time $\turnstile\Neg\Plus p$ and $\turnstile\Neg\Neg p$; and Kolmogorov is right
in that one cannot have these two at the same time if $p$
is as in any of Heyting's two examples, or more generally if $p$ is
a proposition either of type ($\alpha$) or of type ($\beta$).

If we try extract precise definitions from Kolmogorov's letter,
propositions of type ($\alpha$) satisfy $\turnstile\Neg\Neg p\to\Plus p$
and presumably no other identities (we assume, following Heyting, that
{\it all} propositions satisfy $\turnstile\Plus p\to p$, so Kolmogorov's
$\turnstile\Neg\Neg p\to p$ is automatic); whereas propositions of
type ($\beta$) satisfy $\turnstile p\to\Plus p$ and presumably no other
identities.
On the standard interpretation of $\Plus$ and $\Neg$, propositions of
type ($\beta$) should then correspond to all open sets, and
propositions of type ($\alpha$) to all sets $S$ such that
$\Int\Cl S\subset\Int S$.
This includes, in particular, all closed sets and all regular open sets,
and no other open sets.

Although this is clearly not what Kolmogorov could have meant
regarding the judgement $\turnstile\Neg\Plus p\land\Neg\Neg p$, which
under the standard interpretation of $\Neg$ and $\Plus$ in QS4 is
identified with $\turnstile\neg\Box F\land\neg\Box\neg F$, it is worth
observing that $\fm{\prin\neg\Box p\land\neg\Box\neg p\,/\,\ab}$ is
an admissible rule for S4 (see \cite{M0}*{Example \ref{int:QS4-admissible}}).

Kolmogorov concluded his first letter to Heyting with the proposal
of two operators from propositions of type ($\alpha$) to problems.
The first operator, {\it Prove the given proposition}, resembles
the restriction of our operator ``$\oc$'', which is
defined on all propositions.
The second operator, {\it Derive a contradiction from the given
proposition}, resembles what we denote by $\oc\neg$.
(Note, however, that to interpret Kolmogorov's first operator $\Neg$
as $\Plus\neg$, one has to extend the domain of his second operator
$\Plus$ to the classical negations of propositions of
type ($\alpha$).)
In this setup, Kolmogorov's propositions of type ($\alpha$)
are precisely those propositions that satisfy $\turnstile\neg\oc\neg p\to\oc p$
(note that one of the two negations is classical and the other one
is intuitionistic).

Heyting's 1932 paper \cite{He1'} contains an elaboration of some ideas in
Kolmogorov's first letter, which appears to agree with our conclusions.
For each formula $F$ of a logical calculus that contains classical propositional
calculus, Heyting introduces the problem $\beta F$ of proving $F$;
and discusses the significance of the problems $\neg\beta F\to\beta\neg F$ and
$\neg\beta\neg F\to\beta F$.

\subsection*{Acknowledgements}
I would like to thank L. Beklemishev, M. Bezem, G. Dowek, A. L. Semyonov and D. Shamkanov
for valuable discussions and useful comments.
An initial part of this work was carried out while enjoying the hospitality and
stimulating atmosphere of the Institute for Advanced Study and its
Univalent Foundations program in Spring 2013.

\subsection*{Disclaimers}

1. Some translations quoted in this paper were edited in order to improve syntactic and 
semantic fidelity.
When emphasis is present in quoted text, it is always original.

2. I oppose all wars, including those wars that are initiated by governments at the time when 
they directly or indirectly support my research. The latter type of wars include all wars 
waged by the Russian state in the last 25 years (in Chechnya, Georgia, Syria and Ukraine) 
as well as the USA-led invasions of Afghanistan and Iraq.

\end{document}